\documentclass[12pt, reqno]{amsart}

\usepackage[margin=2.4cm]{geometry}

\usepackage{amsmath,amsfonts,amssymb,amsthm,mathrsfs,microtype}
\usepackage{bbm}
\usepackage{hyperref,graphicx}
\usepackage{color}
\usepackage{cite}

\theoremstyle{plain}
\newtheorem{theorem}{Theorem}[section]
\newtheorem{lemma}[theorem]{Lemma}
\newtheorem{corollary}[theorem]{Corollary}

\theoremstyle{definition}

\newtheorem{remark}[theorem]{Remark}
\newtheorem{notation}[theorem]{Notation}

\DeclareMathOperator{\dimh}{dim_H}

\DeclareMathOperator{\T}{\mathbb{T}}
\newcommand{\modulo}{\ (\mathrm{mod}\ 1)}

\allowdisplaybreaks[3]

\begin{document}

\title[A note on uniform random covering problems]{A note on uniform random covering problems in metric spaces}

\author{Zhang-nan Hu}
\address{Z.-N.~Hu, College of Science, China University of Petroleum, Beijing 102249, P. R. China}
\email{hnlgdxhzn@163.com}

\author{Bing Li*}\thanks{* Corresponding author} 
\address{B. Li, School of Mathematics, South China University of Technology, Guangzhou, 510641, P. R. China}
\email{scbingli@scut.edu.cn}
	
\author{YiJing Wang}
\address{Y.-J. Wang, School of Mathematics, South China University of Technology, Guangzhou, 510641, P. R. China}
\email{1915317901@qq.com}
%王怡静

\date{\today}

\subjclass[2020]{}
	
\setlength{\footskip}{18pt}

\begin{abstract}
In this paper, we study the uniform random covering problem in general metric space $(X,d)$. Let $\omega=(\omega_n)_{n\in\mathbb N}$ be a sequence of independent identically distributed random variables on $(X,\mu)$, and $\ell=(l_n)_{n\in\mathbb N}$ a sequence of positive real numbers. We analyze  the size of the set
\[\mathcal{U}(\omega,\ell)=\left\{y\in X\colon \forall N\gg1,~\exists 1\le n\le N,~s.t. ~d(\omega_n,y)<\ell_N\right\},\]
and establish the 0-1 law for the Hausdorff dimension of $\mathcal{U}(\omega,\ell)$, its measure and the event $\mathcal{U}(\omega,\ell)=X$. Some sufficient conditions are provided  for  $\mathcal{U}(\omega,\ell)$ to have full measure or be countable almost surely. Furthermore, we employ the local dimension of $\mu$ to estimate the Hausdorff dimension of $\mathcal{U}(\omega,\ell)$. %The uniform random covering problem in the   torus $\mathbb{T}$  has previously been studied by  Koivusalo, Liao and Persson[Int. Math. Res. Not. 2023 (2023) 455–481].
 While prior work by  Koivusalo, Liao and Persson ( Int. Math. Res. Not. 2023) addressed the case of the torus $\mathbb{T}$, we apply our results to the $d$-dimension  torus $\mathbb{T}^d$, and explicit analysis of the Hausdorff dimension in a critical case is given.
\end{abstract}

\maketitle

\section{Introduction}
The classical problem of Diophantine approximation originates in theory of rational approximations to real numbers. A well-known result in this area is Dirichlet's theorem, established by  Dirichlet (1842):
\begin{theorem}[Dirichlet's theorem]
For real numbers  $\xi$ and $N > 1$, there exists an integer with $1\le n<N$, such that
\[\Vert n\xi \Vert\le N^{-1},\]
where $\Vert\cdot\Vert$ denotes the distance to the nearest integer.
\end{theorem}
%Throughout this text, $\Vert\cdot\Vert$ denotes the distance to the nearest integer. In  The well-known Dirichlet's theorem states that for any real numbers $\xi$ and $N \ge 1$, there exists an integer $1 \le n \le N$, such that 
%Let $\mathbb{T}=\mathbb{R}/\mathbb{Z}$ be the one-dimensional torus. We consider $T_\xi\colon \mathbb{T}\to\mathbb{T}\colon$ \[T_\xi x=x+\xi\modulo.\]
%Then from the perspective of dynamical systems,  $(\Vert n\xi \Vert)_{n\in\mathbb{N}}$ is the orbit of $0$ of $T_\xi$, and Dirichlet's theorem quantifies the rate of uniform approximation of $0$ by the orbit $(\Vert n\xi \Vert)_{n\in\mathbb{N}}$.
% the rate at which $0$ is approximated by the orbit $(\Vert n\xi \Vert)_{n\in\mathbb{N}}$ in a uniform way.
This result implies that the sequence $(\Vert n\xi \Vert)_{n\in\mathbb{N}}$ approximates $0$ with a polynomial speed of degree one in a uniform way, which is called uniform approximation (with respect to $N$) in \cite{wald} by Waldshmidt. %Let $\mathbb{T} = \mathbb{R}/\mathbb{Z}$ denote the one-dimensional torus. 
Kim and Liao \cite{KL} studied an inhomogeneous version of Dirichlet’s theorem, replacing  $0$ by an arbitrary point $y$ and allowing  a faster approximation speed. For a fixed irrational $\xi$ and $\alpha>0$, they calculated the Hausdorff dimension of the following set
\[\mathcal{U}[\xi,\alpha]:=\{y\in \mathbb{T}\colon \forall N\gg1,~\exists 1\le n\le N,~s.t. ~\Vert n\xi-y\Vert<N^{-\alpha}\},\]
which is shown to depend on the irrationality exponent of $\xi$. 

Let $\mathbb{T} = \mathbb{R}/\mathbb{Z}$ denote the one-dimensional torus.  Consider the translation map $T_\xi \colon \mathbb{T} \to \mathbb{T}$ defined by
\[T_\xi x=(x+\xi)\modulo.\]
Then  from the perspective of dynamical systems, the sequence $(\Vert n\xi \Vert)_{n\in\mathbb{N}}$ corresponds to the orbit of $0$ under the map of $T_\xi$. Naturally, one may consider the corresponding set when the orbit $(T_\xi^n0)_{n\in\mathbb{N}}$  is replaced by the orbit $(T^nx)_{n\in\mathbb{N}}$ of a general dynamical system $(\T, T )$, see \cite{BL16, KKP20, FST,HLiao24}. Moreover,  $\Vert n\xi \Vert$ can also be expresses as $\Vert n\xi \Vert= T_\xi^n x-x$, which leads to a uniform recurrence problem, see \cite{LPS, ZWu}.

%Instead of considering the sequence $(\Vert n\xi \Vert){n\in\mathbb{N}}$ in the context of a general dynamical system, 
Motivated by the studies of the uniform approximation problem, 
Koivusalo, Liao and Persson \cite{KLP} replaced the deterministic sequence $(\Vert n\xi \Vert)_{n\in\mathbb{N}}$ with a sequence $\omega := (\omega_n)_{n\in\mathbb{N}}$ of independent, identically distributed (i.i.d. for short) random variables uniformly distributed on  $\T$  with respect to Lebesgue measure. More precisely, given a decreasing sequence of positive numbers $(\ell_n)_{n\in\mathbb{N}}$, they  \cite{KLP} introduced the uniform random covering set
%Instead of replacing the sequence $(\Vert n\xi \Vert)_{n\in\mathbb{N}}$ by an orbit of a general dynamical system, Koivusalo et al. \cite{KLP} studied the  analog results when $(\Vert n\xi \Vert)_{n\in\mathbb{N}}$ is replaced by an independent and identically distributed (i.i.d.) sequence  $\omega:=(\omega_n)_n \subset \T $ which are uniformly distributed on $\T$ with respect to Lebesgue measure. Let $(\ell_n)_n$ be a decreasing sequence of positive numbers. Koivusalo et al. \cite{KLP} investigate the size, in the sense of both Lebesgue measure and Hausdorff dimension, of the following set
\[\mathcal{U}(\omega)=\left\{y\in \T\colon \forall N\gg1,~\exists 1\le n\le N,~s.t. ~\Vert \omega_n-y\Vert <\ell_N\right\}.\]
%Such sets are called uniform random covering sets.
and investigated its size, in terms of both Lebesgue measure and Hausdorff dimension. %Moreover, they posed some open questions: 
%\begin{itemize}
%\item {\bf Zero-one law for full covering}: Does the event $\{\mathcal{U}(\omega) = \mathbb{T}\}$ obey a zero-one law?
%\item {\bf Zero-one law for dimension}: Does the Hausdorff dimension of $\mathcal{U}(\omega)$ satisfy a zero-one law?
%\end{itemize}
%\medskip

%\noindent\textbf{(1) Zero-one law for full covering:} 
%\[\text{{\it Does the event }}\{\mathcal{U}(\omega) = \mathbb{T}\}\text{ {\it obey a zero-one law?}
%}\]

%\noindent\textbf{(2) Zero-one law for dimension:} 
%\[\text{{\it Does the Hausdorff dimension of }}\mathcal{U}(\omega)\text{ {\it satisfy a zero-one law?}
%}\]
%\medskip
%These questions arise naturally and reveal a deeper probabilistic structure in covering problems.

The uniform covering problem is closely related to, but different from, the famous random covering problem. Let $(\omega_n)_{n\in\mathbb{N}}$ be an i.i.d. random sequence with uniform distribution on $\T$.  In 1956, Dvoretzky \cite{dvor} proposed the question when all points in $\T$ are covered by many open intervals $I_n:=(\omega_n-\ell_n,\omega_n+\ell_n)$ (modulo 1)  infinitely often with probability 1. In 1972, Shepp \cite{shep} provided  a necessary and sufficient condition:  $\T=\limsup\limits_{n\to\infty}I_n$ almost surely  if and only if
\[\sum_{n=1}^\infty \frac{1}{n^2}\exp\left(2\sum_{k=1}^n\ell_k \right) = \infty.\]
Applying Borel-Cantelli lemma and  Fubini's theorem, the Lebsgue measure of $\limsup\limits_{n\to\infty}I_n$ is 0 or 1 almost surely according
to the convergence or divergence of $\sum_{n=1}^\infty\ell_n$. Let $a>0$ and $\tau>1$. When $\ell_n=\frac{a}{n^\tau}$, Fan and Wu \cite{FW} showed that the Hausdorff dimension of $\limsup\limits_{n\to\infty}I_n$ is $1/\tau$ almost surely.  Durand \cite{durand} gave the Hausdorff dimension  of $\limsup\limits_{n\to\infty}I_n$ for general  sequences $(\ell_n)_n$.
%This problem attracted much attention of mathematicians, and addressed  many variations of the covering problems. We refer \cite{Ka1,Ka2} for a history of and see \cite{EJJS,Persson,FW,Feng,HL,HLX,Jarvenpa} for more recent developments.
For extensive surveys and recent advances on Dvoretzky covering problem, we refer to \cite{Ka1,Ka2,EJJS,Persson,FW,Feng,HL,HLX,Jarvenpa}.

Motivated by these developments, we would like to investigate the uniform random covering problem in the context of a general metric space.  Let $(X,d)$ be a metric space, and let $\omega:=(\omega_n)_{n\ge1}$ be a sequence of independent random variables defined on a probability space $(\Omega,\mathscr{A},\mathbb{P})$, taking values in  $X$ and having the same distribution, denoted by $\mu$,  which is the probability measure defined by 
\begin{equation}\label{eqmu}
	\mu(A)=\mathbb{P}(\omega_1\in A)
\end{equation}
for all Borel sets $A \subset X$. We assume that throughout that $\mu$ is non-atomic, that is%In this paper, we always suppose that  for any $x\in X$ and $r\ge 0$
\begin{equation}\label{mu}
\mu(\{y\in X\colon d(y,x)=r\})=0\quad \text{for~all~}x\in X\text{~and~}r\ge0.
\end{equation}
%Suppose that $\omega:=(\omega_n)_{n\ge1}$ is a sequence of independent random variables which are uniformly distributed on $X$.
%Let $(\omega_n)_{n\ge1}$ be a stationary process defined on a probability space $(\Omega,\mathcal{B},\mathbb{P})$ and take values in a compact metric space $(X,d)$. Let $\mu$ be the the distribution of $\omega_1$ which is the probability measure defined by 
%\begin{equation}\label{eqmu}
%	\mu(A)=\mathbb{P}(\omega_1\in A)
%\end{equation}
%for all Borel sets $A \subset X$. 
%The measure $\mu$ is the distribution of $\xi_1$.
%\begin{definition}\label{exp}
%We say that $(\omega_n)_{n\ge1}$ is $exponentially~ mixing$ if there exist two constants $C>0$ and $0<\rho<1$
%such that  
%\[\big|\mathbb{P}(\omega_1\in A|D)-\mathbb{P}(\omega_1\in A)\big|\le C\rho^{n}
%\]
%for all $n\ge1$, balls $A\subset X$ and $D\in\mathcal{B}^{n+1}$, where $\mathcal{B}^{n+1}$ is the sub-$\sigma$-field 
%generated by $(\omega_{n+i})_{i\ge1}$.
%\end{definition}
Let  $\ell:=(\ell_n)_{n\ge1}$ be a sequence of positive numbers decreasing to zero. We define the corresponding {\it uniform random covering set} to be 
\begin{equation*}
\begin{split}
\mathcal{U}(\omega,\ell)&=\left\{y\in X\colon \forall N\gg1,~\exists 1\le n\le N,~s.t. ~d(\omega_n,y)<\ell_N\right\}\\
&=\bigcup_{p=1}^{\infty}\bigcap_{N=p}^{\infty}\bigcup_{n=1}^NB(\omega_n,\ell_N). 
\end{split}
\end{equation*}
As the studies in torus $\T$, we are interested in the metric properties (the measure $\mu$ and the Hausdorff dimension) of the set $\mathcal{U}(\omega,\ell)$. %—specifically, the measure $\mu$, and the Hausdorff dimension—of the set $\mathcal{U}(\omega,\ell)$.

The remainder of the article is organized as follows.  Section 2 formally states our main results concerning measure and dimension. Section 3 contains the proof of the zero-one laws, that is, Theorem 2.1, while Section 4 gives the proofs of Theorems 2.2-2.3. In Section 5, we first investigate the hitting times and the connection to the local dimension of the measure $\mu$. We then use the study to provide the proof of Theorem 2.4. Finally, Section 6 presents an application to $d$-dimensional torus, and we furthermore investigate the Hausdorff dimension at a critical case.

\section{Main results}

The following theorem %gives a positive answer to the questions listed above. More precisely, this theorem 
proves the existence of a 0-1 law for each of the following: the Hausdorff dimension of $\mathcal{U}(\omega,\ell)$, the measure of $\mathcal{U}(\omega,\ell)$, and the event $\{\mathcal{U}(\omega,\ell) = X\}$.
\begin{theorem}\label{theorem1}
Assume that the sequence $\ell=(\ell_n)_n$ is decreasing. Let $0\le s\le \dimh X$, $0\le t\le 1$. Let
\[A:=\{\mathcal{U}(\omega,\ell)=X\},\quad B:=\{ \dimh \mathcal{U}(\omega,\ell)=s\},\quad C:=\{\mu(\mathcal{U}(\omega,\ell))=t\}.\]
Then $\mathbb{P}(A),\mathbb{P}(B), \mathbb{P}(C)\in\{0,1\}.$
\end{theorem}
%{\color{red} Question: the results above only depend on $\ell$?}

For $\ell=(\ell_n)_n$, denote
\[E(\ell):=\left\{y\in X\colon \sum_{n=1}^\infty\mu(B(y,\ell_n))=\infty~{\rm and~}\sum_{n=1}^\infty \mu(B(y,\ell_n))e^{-n\mu(B(y,\ell_n))}<\infty\right\}.\]
%{\color{red}\[E(\ell):=\left\{y\in X\colon \sum_{n=1}^\infty \mu(B(y,\ell_n))e^{-n\mu(B(y,\ell_n))}<\infty\right\}.\]}
\begin{theorem}\label{theorem2}
Assume that the sequence $\ell=(\ell_n)_n$ is decreasing, and for all $y\in X$, the sequence $(n\mu(B(y,\ell_n)))_n$ is nondecreasing. Then almost surely
\[\mu(\mathcal{U}(\omega,\ell))=\mu(\mathcal{U}(\omega,\ell)\cap E(\ell))=\mu(E(\ell)).\]
\end{theorem}

\begin{theorem}\label{theorem3}
Assume that the sequence $\ell=(\ell_n)_n$ is decreasing. If 
\[\sum_{n=1}^\infty n\int_X\mu(B(y,\ell_n+\ell_{n+1}))d\mu(y)<\infty,\]
then almost surely $\mathcal{U}(\omega,\ell)=\{\omega_k\colon k\in\mathbb{N}\}.$
\end{theorem}

Let us introduce a notion of the local dimension of measure $\mu$. The {\it upper} and {\it lower local dimensions} of a measure $\mu$ at a point $y \in {\rm supp}\mu$ are defined as
 \begin{align*}
\overline{d}_\mu(y) = \limsup_{r\searrow 0} \frac{\log\mu(B(y,r))}{\log r} \quad \text{and} \quad \underline{d}_\mu(y)  = \liminf_{r\searrow 0} \frac{\log\mu(B(y,r))}{\log r},
\end{align*}
where $B(y,r)$ is the  ball of center $y$ and radius $r$. If $\overline{d}_\mu(y) =\underline{d}_\mu(y),$ we refer to the common value as the {\it local dimension} of  $\mu$ at   $y \in {\rm supp}\mu$
\[ d_\mu(y)  = \lim_{r\searrow 0} \frac{\log\mu(B(y,r))}{\log r}.\]
If $y\notin {\rm supp}\mu$, let $\overline{d}_\mu(y) ,\, \underline{d}_\mu(y)=\infty. $

The following theorem shows that the Hausdorff dimension of $ \mathcal{U}(\omega,\ell)$ is related to the upper and lower local dimensions of  measure $\mu$.

\begin{theorem}\label{theorem4}
Assume that $(X,d)$ is a separable complete metric space.     Let  $\ell=(n^{-\alpha})_n$,  $\alpha>0$. Then almost surely
\[\dimh\left\{y\in X\colon \overline{d}_\mu(y)<\frac{1}{\alpha}\right\}\le \dimh \mathcal{U}(\omega,\ell)\le \dimh\left\{y\in X\colon \underline{d}_\mu(y)\le\frac{1}{\alpha}\right\}.\]
\end{theorem}

\section{Proof of Theorem \ref{theorem1}}
The following conclusion is a variation of Kolmogorov's zero-one law. It states that for events satisfying specific asymptotic conditions in a sequence of independent random variables, the event is almost certain, that is, has probability 0 or 1.
\begin{theorem}\label{thm01}
Let $X_1,X_2,\cdots$ be a sequence of independent random variables. For $A\in\sigma(X_1,X_2,\cdots),$ if for any $n,$ there exists $A_n\in\sigma(X_n,X_{n+1},\cdots)$ such that $\mathbb{P}(A\,\Delta\,A_n)=0,$ then 
$$\mathbb{P}(A)=0\ \text{or}\ 1.$$
\end{theorem}

\begin{proof}[Proof of Theorem \ref{theorem1}]
For any integer $k\geq1$,  denote
$$\mathcal{U}_{k}(\omega,\ell):=\big\{y\in X:\forall \,N\gg1,\,\exists\, \, n\in\{k,\cdots,N\},\,\,\,\text{s.t.}\  d(\omega_n,y)<\ell_N\big\}.$$
Then $\big(\mathcal{U}_{k}(\omega,\ell)\big)^{c}$ can be rewritten as
  $$\big\{y\in X:\exists \, \text{infinitely~many}\ N ,\ \text{such ~that~}\ \forall \,n\in\{k,\cdots,N\},\   d(\omega_n,y)\geq \ell_N\big\}.$$
For any $y\in \mathcal U(\omega,\ell)\setminus \mathcal U_{k}(\omega,\ell)$, there exists infinitely many  $N\geq1$ such that
  \[
  \min_{1\leq n\leq k-1}d(\omega_n,y)<\ell_N\,.
  \]
Since $(\ell_n)_{n\geq 1}$ is decreasing to 0, it gives that
  \[
  \min_{1\leq n\leq k-1}d(\omega_n,y)\leq\liminf_{N\to\infty}\ell_N=0\,.
  \]
  Then
  \[
  \mathcal U(\omega,\ell)\setminus \mathcal U_{k}(\omega,\ell)\ \subseteq\ \{\omega_1,\ldots,\omega_{k-1}\}\,.
  \]
Note that $\mathcal U_{k}(\omega,\ell) \subseteq \mathcal U(\omega,\ell)$. It follows that 
  \begin{align}\label{for1}
  \mathcal U_{k}(\omega,\ell)\ \subseteq\  \mathcal U(\omega,\ell)\ \subseteq\  \mathcal U_{k}(\omega,\ell)\cup\{\omega_1,\ldots,\omega_{k-1}\}.
  \end{align}

\noindent$\bullet$ \emph{$0$-$1$ law of the event $A$}.  
Given $x \in X$ and $k\geq1$, for $\{x\in\mathcal U(\omega,\ell)\}$ and $\{x\in\mathcal U_{k}(\omega,\ell)\}$, we have that
  \begin{align*}
  &~~~~\mathbb{P}\big(\big\{x\in\mathcal U(\omega,\ell)\big\}\,\Delta\, \big\{x\in\mathcal U_{k}(\omega,\ell)\big\}\big)\leq\mathbb{P}\big(x\in\{\omega_1,\ldots,\omega_{k-1}\}\big)\\
    &=\mathbb{P}\left(\bigcup\limits_{i=1}^{k-1}\{x=\omega_i\}\right)\leq \sum_{i=1}^{k-1} \mathbb{P}\big(\{x=\omega_i\}\big)\\
    &=(k-1)\cdot\mu(\{x\})=0\,.
  \end{align*}  
We observe that
\[
\big\{x\in\mathcal U_{k}(\omega,\ell)\big\}\in\sigma(X_k,X_{k+1},\cdot\cdot\cdot)\,.
\]
Combing Theorem~\ref{thm01}~, we get that for $x\in X$,  $\mathbb{P}\big(x\in\mathcal U(\omega,\ell)\big)=0$ or $1\,$.
 
 Therefore there are two cases:
   \medskip
  
  \emph{Case $1:$ there exists $x_0\in X$ such that $\mathbb{P}\big(x_0\in\mathcal U(\omega,\ell)\big)=0\,$.}
  Since 
  $$\{\mathcal U(\omega,\ell)=X\}\subset\{x_0\in \mathcal U(\omega,\ell)\},$$
  we have $\mathbb{P}( A) = 0.$ 
  \medskip

  \emph{Case $2:$ for any  $x\in X$, $\mathbb{P}\big(x\in\mathcal U(\omega,\ell)\big)=1$.} 
 By~\eqref{for1},
 $$\big\{x\in\mathcal{U}(\omega,\ell)\big\}=\big\{x\in\{\omega_1,\ldots,\omega_{k-1}\}\big\}\cup\big\{x\in\mathcal{U}_{k}(\omega,\ell)\big\},$$
thus
\begin{eqnarray}\label{ppxukoeq1}
    \mathbb{P}\big(x\in \mathcal U_{k}(\omega,\ell)\big)=1,\quad \forall\,x\in X,\ \ \forall\, k\geq1.
\end{eqnarray}

\iffalse
Assume that ${\mathcal N}\subseteq\mathbb{N}$, define a projection map by
\[
\pi_{{\mathcal N}}:X^{\mathbb{N}}\to X^{{\mathcal N}}:=\prod_{i\in{\mathcal N}}X.
\]
In particular, for $k\in\mathbb{N}$, let
\[
\pi_{1\leq i\leq k}:=\pi_{\{1,2,...,k\}},\qquad\pi_{i=k}:=\pi_{\{k\}},\qquad \pi_{i\geq k}:=\pi_{\{k,k+1,...\}},
\]
Define $ \mathbb{P}_{{\mathcal N}}:= \mathbb{P}\circ\pi_{{\mathcal N}}^{-1}$, and for $k\in\mathbb{N}$, let
\[
 \mathbb{P}_{1\leq i\leq k}:= \mathbb{P}_{\{1,2,...,k\}},\qquad \mathbb{P}_{i=k}:= \mathbb{P}_{\{k\}},\qquad \mathbb{P}_{i\geq k}:= \mathbb{P}_{\{k,k+1,...\}}\,.
\]
Then we rewrite \eqref{ppxukoeq1} as
\begin{eqnarray*}
     \mathbb{P}_{i\geq k}\Big(\pi_{i\geq k}\big\{\omega\in\Omega:x\in\mathcal{U}_k(\omega,\{r_n\})\big\}\Big)=1,\qquad\forall\,x\in X,\ \ \ \forall\, k\geq1.
\end{eqnarray*}
\fi
Put $\mathcal{F}_{k-1}:=\sigma(\omega_1,\dots,\omega_{k-1})$.
Therefore for $k\in\mathbb{N}$ and $1\leq j\leq k-1$,
\begin{eqnarray*}
   \mathbb{P}\big(\omega_j\in\mathcal{U}_k(\omega,\ell)\big)&=& \mathbb{E}\big(\mathbb{E}(1_{\{\omega_j\in \mathcal{U}_k(\omega,\ell)\}}|\mathcal{F}_{k-1})\big)=1.
   %\int \mathbb{P}_{i\geq k}\Big(\pi_{i\geq k}\big\{\omega\in\Omega:\omega_j=\omega_j^{'},\,\omega_j\in\mathcal{U}_k(\omega,\{r_n\})\big\}\Big)\,\mathrm{d} \mathbb{P}_{i=j}(\omega_j^{'})\\
%    &=&\int \mathbb{P}_{i\geq k}\Big(\pi_{i\geq k}\big\{\omega\in\Omega:x\in\mathcal{U}_k(\omega,\{r_n\})\big\}\Big)\,\mathrm{d}\mu(x)\\
 %   &=&1\,.
\end{eqnarray*}
%\begin{eqnarray*}
 %    \mathbb{P}\big(\omega_i\in\mathcal{U}_k(\omega,\{r_n\})\big)&=&\int \mathbb{P}_{i\geq k}\Big(\pi_{i\geq k}\big\{\omega\in\Omega:\omega_j=\omega_j^{'},\,\omega_j\in\mathcal{U}_k(\omega,\{r_n\})\big\}\Big)\,\mathrm{d} \mathbb{P}_{i=j}(\omega_j^{'})\\
%    &=&\int \mathbb{P}_{i\geq k}\Big(\pi_{i\geq k}\big\{\omega\in\Omega:x\in\mathcal{U}_k(\omega,\{r_n\})\big\}\Big)\,\mathrm{d}\mu(x)\\
 %   &=&1\,.
%\end{eqnarray*}
It follows that 
\begin{align}\label{for2}
  \mathbb{P}\Big(\{\omega_1,\ldots,\omega_{k-1}\}\subseteq \mathcal U_{k}(\omega,\ell)\Big)=1\,.
  \end{align}
Note that $\{\mathcal U_k(\omega,\ell) = X\}\subset A=\{\mathcal U(\omega,\ell) = X\}$, then %For $x\in X,$ %考虑事件 $A$ 与 $\{\mathcal U_k(\omega,\{r_n\}) = X\}$的对称差:
  \begin{align*} 
   & \ \ \ \ \mathbb{P}(A\ \Delta\ \{\mathcal U_k(\omega,\ell) = X\})=\mathbb{P}\big(A\backslash\{\mathcal U_k(\omega,\ell) = X\})\big)\\
   &\leq \mathbb{P}\big(\{\omega_1,\ldots,\omega_{k-1}\}\not\subseteq \mathcal U_{k}(\omega,\ell)\big).
 \end{align*}
By ~\eqref{for2}, we have
$$\mathbb{P}\big(A\ \Delta\ \{\mathcal U_k(\omega,\ell) =X\}\big)=0\,.$$
Note that for $k\geq1$,
\[
\big\{\mathcal U_k(\omega,\ell) =X\big\}\in\sigma(X_k,X_{k+1},\cdot\cdot\cdot)\,,
\]
It follows from Theorem~\ref{thm01} that $\mathbb{P}(A)=0$ or $1\,$.
\medskip

  \noindent $\bullet$ \emph{ $0$-$1$ law of the event  $B$}. Since $\dim_{\rm H}\{\omega_1,\dots,\omega_{k-1}\}=0$, we obtain that
  $$\left\{\mathrm{dim}_\mathrm{H} \ \mathcal{U}(\omega,\ell) = s_1\right\}=\left\{\mathrm{dim}_\mathrm{H}\ \mathcal{U}_k(\omega,\ell) = s_1\right\},$$  
Thus $$\mathbb{P}\big(\left\{\mathrm{dim}_\mathrm{H} \ \mathcal{U}(\omega,\ell) = s_1\right\}\, \Delta\, \left\{\mathrm{dim}_\mathrm{H} \ \mathcal{U}_k(\omega,\ell) = s_1\right\}\big)=0\,,$$
implying that $\mathbb{P}(B) = 0$ or $1\,$.
\medskip

\noindent $\bullet$ \emph{$0$-$1$ law of the event $C$}. Note that $\mu(\mathcal{U}(\omega,\ell))=\mu(\mathcal{U}_k(\omega,\ell)). $ Then 
$$\mathbb{P}\big(\left\{\mu(  \mathcal{U}(\omega,\ell) )= t_1\right\}\, \Delta\, \left\{\mu(\mathcal{U}_k(\omega,\ell)) = t_1\right\}\big)=0\,,$$
thus $\mathbb{P}(C) = 0$ or $1\,$.
%$\mu$的定义, 可知有限点的测度为0, 于是与事件 $B$的$0$-$1$律类似, 可得$\mathbb{P}(C) = 0$ 或者 $1\,$.
\end{proof}

\section{Proofs of Theorems \ref{theorem2}-\ref{theorem3}}
We will need the following theorem of \cite{galambos} to prove Theorem \ref{theorem2}.

\begin{theorem}[Theorem 4.3.1 of  \cite{galambos}]\label{gala}
Let $Y_1, Y_2, \dots$ be independent, identically distributed random variables with a common, nondegenerate, continuous distribution function $F$. Assume that sequences $(u_n)_{n\geq 1}$ and $(n(1-F(u_n)))_{n\geq 1}$ are both nondecreasing. Let $Z_n=\max \{Y_1, Y_2, \ldots, Y_n\}$. Then the probability
\begin{align*}\label{formula-G}
\mathbb{P}(Z_n\le u_n \ \text{\rm for infinitely many } n)=0,
\end{align*}
if and only if  
\[
\sum_{j=1}^\infty (1-F(u_j))=\infty \quad \text{and}\quad \sum _{j=1}^\infty (1-F(u_j))\exp(-j(1-F(u_j)))<\infty. 
\]
\end{theorem}
It is ready to prove Theorem \ref{theorem2}.

\begin{proof}[Proof of Theorem \ref{theorem2}]
Given $y\in X$ and $n\ge1$, let 
\[Y_n=Y_{n,y}:=(d(\omega_n,y))^{-1}\quad{\rm and}\quad  u_n=l_n^{-1}.\] 
Let $Z_n=\max \{Y_1, Y_2, \ldots, Y_n\}$.
Since $(\omega_n)_{n\geq 1}$ are  independent and identically distributed, the  random variables $(Y_n)_{n\geq 1}$ are also   independent and identically distributed with a common distribution $F$, that is,
\begin{align*}
F(x)&=\mathbb{P}(Y_n<x)=\mathbb{P}(d(\omega_n,y)>x^{-1})\\
&=1-\mathbb{P}(\omega_n\in B(y,x^{-1}))=1-\mu(B(y,x^{-1})),
\end{align*}
which is nondegenerate and continuous due to the assumption \eqref{mu}. Since $1-F(u_n)=\mu(B(y,l_n))$, sequences $(u_n)_{n\geq 1}$ and $(n(1-F(u_n)))_{n\geq 1}$ are both nondecreasing. Note that $y\in E(\ell)$ if and only if
\[\sum_{n=1}^\infty (1-F(u_n))=\sum_{n=1}^\infty \mu(B(y,l_n))=\infty,\]
and
\begin{align*}
&\sum _{n=1}^\infty (1-F(u_n))\exp(-n(1-F(u_n)))\\
=&\sum_{n=1}^\infty\mu(B(y,l_n)) \exp(-n\mu(B(y,l_n)))<\infty.
\end{align*}
Applying Theorem \ref{gala}, for $y\in E(\ell)$,
\[\mathbb{P}(Z_n\le u_n\ \text{\rm for infinitely many } n)=0,\]
and for $y\notin E(\ell)$,
\[\mathbb{P}(Z_n\le u_n\ \text{\rm for infinitely many } n)=1.\]
We observe that $y\in \mathcal{U}(\omega,\ell)$ if and only if  $Z_n>u_n$ eventually, that is, $Z_n\le u_n$ hold for at most finitely many $n\ge1$.
Then we conclude that
\begin{equation*}
\begin{split}
&\int \mu(\mathcal{U}(\omega,\ell)\cap E(\ell))d \mathbb{P}=\int\int 1_{\mathcal{U}(\omega,\ell)}(y)1_{E(\ell)}(y)d\mu(y)d\mathbb{P}\\
=&\int_{E(\ell)}\left(\int 1_{\mathcal{U}(\omega,\ell)}(y)d\mathbb{P}\right)d\mu(y)=\int 1_{E(\ell)}(y)d\mu(y)=\mu(E(\ell)),
\end{split}
\end{equation*}
and 
\begin{equation*}
\begin{split}
&\int \mu(\mathcal{U}(\omega,\ell)\cap E(\ell)^c)d \mathbb{P}=\int\int 1_{\mathcal{U}(\omega,\ell)}(y)1_{E(\ell)^c}(y)d\mu(y)d\mathbb{P}\\
=&\int_{E(\ell)^c}\left(\int 1_{\mathcal{U}(\omega,\ell)}(y)d\mathbb{P}\right)d\mu(y)=0.
\end{split}
\end{equation*}
Thus 
\begin{equation*}
\begin{split}
\int \mu(\mathcal{U}(\omega,\ell))d \mathbb{P}&=\int \mu(\mathcal{U}(\omega,\ell)\cap E(\ell))d \mathbb{P}+\int \mu(\mathcal{U}(\omega,\ell)\cap E(\ell)^c)d \mathbb{P}\\
&=\int \mu(\mathcal{U}(\omega,\ell)\cap E(\ell))d \mathbb{P}=\mu(E(\ell)),
\end{split}
\end{equation*}
which implies that with probability one, $\mu(\mathcal{U}(\omega,\ell))= \mu(\mathcal{U}(\omega,\ell)\cap E(\ell))=\mu(E(\ell))$ holds.
\end{proof}

\begin{proof}[Proof of Theorem \ref{theorem3}]

For $n\ge1$ and $x\in [0,1]$, we have $(1-x)^n\ge 1-nx$, then 
\begin{equation*}
\begin{split}
& \mathbb{P}\left(B(\omega_{n+1},l_{n+1})\cap B(\omega_k,l_n)=\emptyset ~{\rm for ~all}~k\le n\right)\\
=&\mathbb{P}\left(\bigcap_{k=1}^n\{\omega_k\notin B(\omega_{n+1},l_n+l_{n+1})\}\right)\\
=&\idotsint\prod_{k=1}^n(1-1_{B(\omega_{n+1},l_n+l_{n+1})}(\omega_k)) d \mathbb{P}(\omega_1)\cdots d \mathbb{P}(\omega_{n+1})\\
=&\int \prod_{k=1}^n(1-\mu(B(\omega_{n+1},l_n+l_{n+1})))d \mathbb{P}(\omega_{n+1})\\
\ge&\int (1-n\mu(B(\omega_{n+1},l_n+l_{n+1})))d \mathbb{P}(\omega_{n+1})\\
=&1-n\int \int 1_{B(\omega_{n+1},l_n+l_{n+1})}(y)d \mu(y)d \mathbb{P}(\omega_{n+1})\\
=&1-n\int \mu(B(y,l_n+l_{n+1})))d \mu(y).
\end{split}
\end{equation*} 
Thus 
\[ \mathbb{P}\left(B(\omega_{n+1},l_{n+1})\cap \bigcup_{k=1}^nB(\omega_k,l_n)\ne\emptyset\right)\le n\int \mu(B(y,l_n+l_{n+1})))d \mu(y).\]
Since $\sum_{n=1}^\infty n\int \mu(B(y,l_n+l_{n+1})))d \mu(y)<\infty$, then by Bore-Cantelli lemma, almost surely, there is some
$m\ge1$ such that for all $n\ge m$, $B(\omega_{n+1},l_{n+1})$ do not intersect $\bigcup_{k=1}^nB(\omega_k,l_n)$. 
Then almost surely,  for $n\ge m$
\begin{equation*}
\bigcup_{k=1}^nB(\omega_k,l_n)\cap \bigcup_{k=1}^{n+1}B(\omega_k,l_{n+1})=\bigcup_{k=1}^nB(\omega_k,l_{n+1}),
\end{equation*}
which implies that for $p\ge m$
\[\bigcap_{n=p}^{\infty}\bigcup_{k=1}^nB(\omega_k,l_n)=\{\omega_k:1\le k\le p\}.\]
Therefore almost surely,
\[\mathcal{U}(\omega,\ell)=\bigcup_{p=m}^\infty\bigcap_{n=p}^{\infty}\bigcup_{k=1}^nB(\omega_k,l_n)=\{\omega_k:k\ge1\}.\]
\end{proof}

\section{Proof of Theorem \ref{theorem4}}
\begin{notation}
For $E\subset X$ and $\delta>0$, $E_\delta=\{y\in X\colon d(x,y)<\delta, x\in E\}.$ 
\end{notation}

\subsection{The upper bound on $\dimh \mathcal{U}(\omega,\ell)$}

\hfill \\

Note that 
\[\left\{y\in X\colon \underline{d}_\mu(y)>\frac{1}{\alpha}\right\}=\bigcap_{m=1}^\infty\left\{y\in X\colon  \underline{d}_\mu(y)>\frac{1}{\alpha}+\frac{1}{m}\right\}.\]
For $m\ge1$, $b\in(\frac{1}{\alpha},\frac{1}{\alpha}+\frac{1}{m})$, and $n\ge1$, put
\[A_{b,n}=\{y\in X\colon \mu(B(y,r))<r^b~{\rm for~all}~r<2^{-n}\}.\]
By definition of $\underline{d}_\mu(y)$ and $A_{b,n}$, 
\begin{equation}\label{abn}
\left\{y\in X\colon  \underline{d}_\mu(y)>\frac{1}{\alpha}+\frac{1}{m}\right\}\subset \bigcup_{n=1}^\infty A_{b,n}.
\end{equation}

Now we claim that with probability one, 
\begin{equation}\label{upper}
\dimh (\mathcal{U}(\omega,\ell)\cap A_{b,n} )=0
\end{equation}
holds for all $b,~n$. Then almost surely
\begin{equation}\label{upper1}
\begin{split}
&\dimh \left(\mathcal{U}(\omega,\ell) \cap\left\{y\in X\colon \underline{d}_\mu(y)>\frac{1}{\alpha}\right\} \right)\\
=&\dimh \left(\mathcal{U}(\omega,\ell) \cap \bigcap_{m=1}^\infty\left\{y\in X\colon  \underline{d}_\mu(y)>\frac{1}{\alpha}+\frac{1}{m}\right\} \right)\\
\overset{\eqref{abn}}{\le}& \dimh \left(\mathcal{U}(\omega,\ell) \cap \bigcup_{n=1}^\infty A_{b,n} \right)\\
=&\sup_{n\ge1} \dimh \left(\mathcal{U}(\omega,\ell) \cap A_{b,n} \right)\overset{\eqref{upper}}{=} 0.\\
\end{split}
\end{equation}
It follows that almost surely
\begin{equation*}\label{upper2}
\begin{split}
\dimh \left(\mathcal{U}(\omega,\ell)\right)&=\dimh \Big(\mathcal{U}(\omega,\ell) \cap \Big(\left\{y\in X\colon \underline{d}_\mu(y)>\frac{1}{\alpha}\right\} \\
&\quad \cup \left\{y\in X\colon \underline{d}_\mu(y) \le\frac{1}{\alpha}\right\}\Big )\Big)\\
&\overset{\eqref{upper1}}{=}\dimh \left(\mathcal{U}(\omega,\ell) \cap \left\{y\in X\colon  \underline{d}_\mu(y)\le \frac{1}{\alpha}\right\} \right)\\
 &\le \dimh \left\{y\in X\colon  \underline{d}_\mu(y)\le \frac{1}{\alpha}\right\},
\end{split}
\end{equation*}
which gives the upper bound in theorem.

Now we show the claim, that is,
\[\dimh (\mathcal{U}(\omega,\ell)\cap A_{b,n} )=0,~\forall b,n\quad a.s..\]
Notice that for any $p\ge1$,
\begin{equation}\label{covering}
\begin{split}
\mathcal{U}(\omega,\ell)&=\bigcup_{l=p}^\infty\bigcap_{N=l}^\infty\bigcup_{k=1}^NB(\omega_k,\ell_N)\subset \bigcup_{l=p}^\infty\bigcap_{j=l}^\infty\bigcup_{k=1}^{2^j}B(\omega_k,\ell_{2^j}).
\end{split}
\end{equation}
then we take $l$ large enough such that $\ell_{2^l}^\alpha+\ell_{2^{l+1}}^\alpha<2^{-n-1}$.

For $i\ge l$, denote
\[G_{l,i}:=\left(\bigcap_{j=l}^i\bigcup_{k=1}^{2^j} B(\omega_k,\ell_{2^j} )\right)\cap A_{b,n},\]
then $G_{l,i}\subset G_{l,i-1}$. 

For $i=l$, denote 
\[I_l:=\{1\le k\le 2^l\colon B(\omega_k,\ell_{2^l} )\cap A_{b,n}\ne\emptyset\},\]
and for $i>l$, denote
\[I_i:=\{1\le k\le 2^i\colon B(\omega_k,\ell_{2^i} )\cap G_{l,i-1}\ne\emptyset\}.\]
We define 
\[\hat{G}_{l,i}:=\bigcup_{k\in I_i}B(\omega_k,\ell_{2^i}),\]
then $G_{l,i}\subset \hat{G}_{l,i}$.
We divide $I_i$ into two parts 
\[J_{1,i}:=\{1\le k\le 2^{i-1}\colon B(\omega_k, \ell_{2^{i-1}})\cap G_{l,i-1}\ne\emptyset\},\]
\[J_{2,i}:=\{2^{i-1}+1\le k\le 2^i\colon B(\omega_k, \ell_{2^i})\cap G_{l,i-1}\ne\emptyset\},\]
then $I_i=J_{1,i}\cup J_{2,i}$. For any $k\in J_{1,i+1}$, since $(\ell_j)$ is decreasing, we get 
\[\emptyset \ne  B(\omega_k, \ell_{2^{i+1}})\cap G_{l,i}\subset  B(\omega_k, \ell_{2^i})\cap G_{l,i-1}, \]
which gives that $k\in I_i$, then we have
\[J_{1,i+1}\subset I_i.\]
Therefore
\begin{equation}\label{estimateofnum}
 \# I_{i+1}\le \#I_i+\#J_{2,i+1}.
 \end{equation}
Let $N_i:=\#I_i$ and $M_i:=\#J_{2,i}$. Now we shall estimate $\mathbb{E}(N_i)$ and $\mathbb{E}(M_i)$. 

By the definition of $M_i$, we get 
\begin{equation*}
\begin{split}
M_i=\sum_{k= 2^{i-1}+1}^{2^i}1_{(G_{l,i-1})_{\ell_{2^i}}}(\omega_k)\le \sum_{k= 2^{i-1}+1}^{2^i}1_{(\hat{G}_{l,i-1})_{\ell_{2^i}}}(\omega_k).
\end{split}
\end{equation*}
Then 
\begin{equation*}
\begin{split}
\mathbb{E}(M_i)\le \int 2^i1_{(\hat{G}_{l,i-1})_{\ell_{2^i}}}(\omega_k)d\mathbb{P}\le \int 2^i\sum_{j\in I_{i-1}}1_{B(\omega_j,\ell_{2^i}+\ell_{2^{i-1}})}(\omega_k)d\mathbb{P}.
\end{split}
\end{equation*}
Since $\{\omega_j\colon j\in I_{i-1}\}$ and $\{\omega_k\colon 2^{i-1}+1\le k\le 2^i\}$  are independent, then 
\begin{equation*}
\begin{split}
\mathbb{E}(M_i)&=\mathbb{E}(\mathbb{E}\left(M_i\mid N_{i-1})\right) \le \mathbb{E}\left(2^i\mathbb{E}\left(\sum_{j\in I_{i-1}}1_{B(\omega_j,\ell_{2^i}+\ell_{2^{i-1}})}(\omega_k)\mid N_{i-1}\right)\right)\\
&=\mathbb{E}\left(2^i\sum_{j\in I_{i-1}}\mu\left(B(\omega_j,\ell_{2^i}+\ell_{2^{i-1}})\right)\right).
\end{split}
\end{equation*}
For $j\in I_{i-1}$, we have $B(\omega_j,\ell_{2^{i-1}})\cap A_{b,n}\ne\emptyset$, implying that there exists $y\in B(\omega_j,\ell_{2^i}+\ell_{2^{i-1}})$ such that 
\[B(\omega_j,\ell_{2^i}+\ell_{2^{i-1}})\subset B(y,2(\ell_{2^i}+\ell_{2^{i-1}})),\]
and by definition of $A_{b,n}$ and the choice of $l$,
\[\mu(B(\omega_j,\ell_{2^i}+\ell_{2^{i-1}}))\le 2^b(\ell_{2^i}+\ell_{2^{i-1}})^b.\] 
It follows that 
\begin{equation*}
\begin{split}
\mathbb{E}(M_i)&\le \mathbb{E}\left(2^iN_{i-1}2^b(\ell_{2^i}+\ell_{2^{i-1}})^b\right)\\
&=2^i 2^b(\ell_{2^i}+\ell_{2^{i-1}})^b\mathbb{E}(N_{i-1})\\
&=2^b(1+2^\alpha)^b2^{(1-\alpha b)i}\mathbb{E}(N_{i-1}).
\end{split}
\end{equation*}
Denote $c_i:=2^b(1+2^\alpha)^b2^{(1-\alpha b)i}\le c_l$. Combining the inequality \eqref{estimateofnum},
\begin{equation*}
\begin{split}
\mathbb{E}(N_i)&\le (1+c_l)\mathbb{E}(N_{i-1})\le (1+c_l)^{i-l} \mathbb{E}(N_{l})\le (1+c_l)^{i-l}2^{l}.
\end{split}
\end{equation*}

Note that for any $\varepsilon>0$,
\begin{equation*}
\begin{split}
\mathbb{P}\left(N_i> (1+\varepsilon)^i (1+c_l)^{i-l}2^{l}\right)&\le \mathbb{P}(N_i>  (1+\varepsilon)^i\mathbb{E}(N_i))\\
&\le \frac{1}{(1+\varepsilon)^i}.
\end{split}
\end{equation*}
Since $\sum_{i=l}^{\infty}\frac{1}{(1+\varepsilon)^i}<\infty$, by Borel-Cantelli lemma, we have
\[\mathbb{P}\left(N_i> (1+\varepsilon)^i (1+c_l)^{i-l}2^{l}\,i.o.\right)=0,\]
which means that, almost surely, there exists $i_0\ge l$ such that $i\ge i_0$
\[N_i \le (1+\varepsilon)^i (1+c_l)^{i-l}2^{l}.\]
Given $p\gg 1$, for $l\ge p$, put
\begin{equation*}\label{covering1}
\begin{split}
G_l:=\bigcap_{j=l}^\infty\bigcup_{k=1}^{2^j}B(\omega_k,\ell_{2^j})\cap A_{b,n}.
\end{split}
\end{equation*}
Note that $G_l\subset G_{l,i}\subset \hat{G}_{l,i}.$
Then $\{ B(\omega_k,\ell_{2^i})\colon k\in I_i\}$ is a covering of $G_l$. It follows that almost surely,
\begin{equation*}
\begin{split}
\dimh G_l&\le \underline{\dim}_{\rm B} G_l\le \liminf_{i\to\infty}\frac{\log N_{\ell_{2^i}}(G_l)}{-\log \ell_{2^i}}\\
&\le \liminf_{i\to\infty}\frac{\log((1+\varepsilon)^i (1+c_l)^{i-l}2^{l})}{-\log \ell_{2^i}}\\
&=\frac{\log(1+\varepsilon)+\log (1+c_l)}{\alpha\log2}\le \frac{\log(1+\varepsilon)+\log (1+c_p)}{\alpha\log2}.
\end{split}
\end{equation*}
It implies that 
\begin{equation*}
\begin{split}
\dimh (\mathcal{U}(\omega,\ell)\cap A_{b,n})\le\dimh \bigcup_{l=p}^\infty G_l&=\sup_{l\ge p}\dimh G_l \le \frac{\log(1+\varepsilon)+\log (1+c_p)}{\alpha\log2},
\end{split}
\end{equation*}
letting $\varepsilon\to0$ and $p\to\infty$, then we obtain that $\dimh (\mathcal{U}(\omega,\ell)\cap A_{b,n})=0$ a.s..

Now we finish the proofs of the claim and the upper bound in Theorem \ref{theorem4}.
\subsection{The lower bound on $\dimh \mathcal{U}(\omega,\ell)$}
\subsubsection{Hitting times}
As we shall see, we use hitting times to investigate the lower bound on $\dimh \mathcal{U}(\omega,\ell)$, and we will also see that hitting times have connection to local dimension of measures.

For a sequence $\omega=(\omega_n)_n$ of random variables, $y\in X$ and $r>0$, the {\it first hitting time} of $B(y,r)$ by $\omega$ is defined as 
\[\tau(\omega,y,r):=\inf\{ n\ge1\colon \omega_n\in B(y,r)\}.\]
For $\omega=(\omega_n)_n$ and $y\in X$, let
\begin{align*}
\overline{H}(\omega,y)&:= \limsup_{r\searrow 0} \frac{\log \tau(\omega,y,r)}{-\log r} \\
\underline{H}(\omega,y)&:= \liminf_{r\searrow 0} \frac{\log \tau(\omega,y,r)}{-\log r} 
\end{align*}
\begin{remark}

If there exists $r_0>0$ such that $B(y,r_0)\cap \{\omega_n:n\ge1\}=\emptyset$, then for any $0<r<r_0$,   $\tau(\omega,y,r)=\infty$, and we let $\overline{H}(\omega,y)=\underline{H}(\omega,y)=\infty.$ 

It worth noticing that in the definitions of $\overline{H}(\omega,y),\,\underline{H}(\omega,y)$, it suffices to consider these limits as $r\to0$ through any decreasing sequence  $(r_n)_n$ with $r_{n+1}/r_n\ge c$ for some constant $0<c<1$.

%If there exists $n_0\ge1$ such that $y=\omega_{n_0}$ if and only if the hitting time sequence 
\end{remark}

\begin{lemma}\label{subset}
For $\alpha>0$, let $\ell=(n^{-\alpha})_n$. Then 
\[\left\{y\in X\colon \overline{H}(\omega,y)<\frac{1}{\alpha}\right\}\subset \mathcal{U}(\omega,\ell)\subset \left\{y\in X\colon \overline{H}(\omega,y)\le \frac{1}{\alpha}\right\}.\]
\end{lemma}
\begin{proof}
First we prove the second inclusion. For $y\in \mathcal{U}(\omega,\ell)$, there exists $m\gg1$ such that for $ j\ge m$, there is some $1\le n \le j$, $y\in B(\omega_n,l_j)$. It implies that 
\[\tau(\omega, y,l_j)\le j.\]
Then we get 
\begin{align*}
\overline{H}(\omega,y)&=\limsup_{j\to\infty}\frac{\log \tau(\omega,y,l_j)}{-\log l_j}\le \frac{1}{\alpha}.
\end{align*}

It remains to prove the first inclusion. Suppose that $\overline{H}(\omega,y)<\frac{1}{\alpha}$.  There exists $r_0>0$ such that for all $0<r<r_0$, we have 
\[\tau(\omega, y,r)<r^{-\frac{1}{\alpha}}.\]
Take $N_0\ge1$ with $l_{N_0}<r_0$. Then for all $N\ge N_0$,
\[n_N:=\tau(\omega,y,l_N)<l_N^{-\frac{1}{\alpha}}=N.\]
The definition of $\tau(\omega,y,r)$ implies that $\omega_{n_N}\in B(y,l_N)$. Thus 
\[y\in \bigcup_{N_0\ge1}\bigcap_{N\ge N_0}\bigcup_{n_N=1}^NB(\omega_{n_N},l_N)=\mathcal{U}(\omega,\ell).\]

\end{proof}

The following lemma shows the relation between the hitting times and the upper local dimension of $\mu$.
\begin{lemma}\label{equality}
Let $y\in X$.  Suppose that $\overline{H}(\omega,y), \overline{d}_\mu(y)<\infty$. Then almost surely we have
\[\overline{H}(\omega,y)=\overline{d}_\mu(y).\] 
\end{lemma}
\begin{proof}
Given $\varepsilon>0$, let $r_n=n^{-\frac{1}{1+\varepsilon}}$, then $\inf\{r_{n+1}/r_n\}>\frac{1}{2}$. It follows from the definition of $\overline{d}_\mu(y)$ that there exists $N\ge1$ such that $n\ge N$
\[\mu(y,r_n)>r_n^{\overline{d}_\mu(y)+\varepsilon}.\]
For $n\ge N$
\begin{align*}
&\mathbb{P}\{\tau(\omega,y,r_n)\ge r_n^{{-\overline{d}_\mu(y)-2\varepsilon}}\}=\mathbb{P}\Big\{\bigcap_{k=1}^{\lfloor r_n^{{-\overline{d}_\mu(y)-2\varepsilon}}\rfloor}\{\omega_k\notin B(y,r_n)\}\Big\}\\
=&\prod_{k=1}^{\lfloor r_n^{{-\overline{d}_\mu(y)-2\varepsilon}}\rfloor}\mathbb{P}\Big\{\{\omega_k\notin B(y,r_n)\}\Big\}=\prod_{k=1}^{\lfloor r_n^{{-\overline{d}_\mu(y)-2\varepsilon}}\rfloor}(1-\mu(B(y,r_n)))\\
<&(1-r_n^{\overline{d}_\mu(y)+\varepsilon})^{\lfloor r_n^{{-\overline{d}_\mu(y)-2\varepsilon}}\rfloor}.
\end{align*}
Note that 
\begin{align*}
&\sum_{n\ge1}\mathbb{P}\{\tau(\omega,y,r_n)\ge r_n^{{-\overline{d}_\mu(y)-2\varepsilon}}\}<\sum_{n\ge1}(1-r_n^{\overline{d}_\mu(y)+\varepsilon})^{\lfloor r_n^{{-\overline{d}_\mu(y)-2\varepsilon}}\rfloor}\\
<&\sum_{n\ge1}e^{-r_n^{\overline{d}_\mu(y)+\varepsilon}\lfloor r_n^{{-\overline{d}_\mu(y)-2\varepsilon}}\rfloor}<\sum_{n\ge1}e^{-r_n^{-\varepsilon-1}}=\sum_{n\ge1}e^{-n}<\infty,
\end{align*}
hence by Borel-Cantelli lemma, almost surely we have
\[\tau(\omega,y,r_n)<r_n^{{-\overline{d}_\mu(y)-2\varepsilon}}\quad {for~}n\gg1,\]
which gives 
\begin{align*}
\overline{H}(\omega,y)&=\limsup_{n\to\infty}\frac{\log \tau(\omega,y,r_n)}{-\log r_n}\\
&\le \limsup_{n\to\infty}\frac{\log r_n^{{-\overline{d}_\mu(y)-2\varepsilon}}}{-\log r_n}=\overline{d}_\mu(y)+2\varepsilon.
\end{align*}
Since $\varepsilon$ is arbitrary, it follows that almost surely $\overline{H}(\omega,y)\le \overline{d}_\mu(y).$
In particular,  if $\overline{d}_\mu(y)=0$, then $\overline{H}(\omega,y)= \overline{d}_\mu(y)=0$. 

Now we assume that $\overline{d}_\mu(y)>0$. Taking $r_n=2^{-n}$, for $\varepsilon>0$, y definition of $\overline{d}_\mu(y)$, there exists a decreasing sequence of integers $(n_j)_j$ such that 
\[\mu(B(y,r_{n_j}))<r_{n_j}^{\overline{d}_\mu(y)-\varepsilon}.\] 
Hence for $j\ge1$
\begin{align*}
&\mathbb{P}\left\{\tau(\omega,y,r_{n_j})\le r_{n_j}^{-\overline{d}_\mu(y)+2\varepsilon} \right\}=\mathbb{P}\Big\{\bigcup_{k=1}^{ \lfloor r_{n_j}^{-\overline{d}_\mu(y)+2\varepsilon}\rfloor }\{\omega_k\in B(y,r_{n_j})\}\Big\}\\
\le &\sum_{k=1}^{ \lfloor r_{n_j}^{-\overline{d}_\mu(y)+2\varepsilon}\rfloor }\mathbb{P}\{\omega_k\in B(y,r_{n_j})\}\le \sum_{k=1}^{ \lfloor r_{n_j}^{-\overline{d}_\mu(y)+2\varepsilon}\rfloor } r_{n_j}^{\overline{d}_\mu(y)-\varepsilon}\\
=&  \lfloor r_{n_j}^{-\overline{d}_\mu(y)+2\varepsilon}\rfloor r_{n_j}^{\overline{d}_\mu(y)-\varepsilon}\le r_{n_j}^{\varepsilon}.
\end{align*}
Since 
\begin{align*}
\sum_{j\ge1}\mathbb{P}\left\{\tau(\omega,y,r_{n_j})\le r_{n_j}^{-\overline{d}_\mu(y)+2\varepsilon} \right\}\le \sum_{j\ge1}r_{n_j}^{\varepsilon}<\infty,
\end{align*}
 Borel-Cantelli lemma implies that almost surely, we have
 \[\tau(\omega,y,r_{n_j})\ge r_{n_j}^{-\overline{d}_\mu(y)+2\varepsilon}, \quad j\ge1.\]
 Then
 \begin{align*}
 \overline{H}(\omega,y)&\ge \limsup_{j\to\infty}\frac{\log \tau(\omega,y,r_{n_j})}{-\log r_{n_j}}\\
&\ge \limsup_{j\to\infty}\frac{\log r_{n_j}^{-\overline{d}_\mu(y)+2\varepsilon}}{-\log r_{n_j}}=\overline{d}_\mu(y)-2\varepsilon.
 \end{align*}
 The arbitrariness of $\varepsilon$ leads to $ \overline{H}(\omega,y)\ge \overline{d}_\mu(y)$.
\end{proof}

\subsubsection{Proof of the lower bound}Now we are ready to prove the lower bound in Theorem \ref{theorem4}, that is,
\[\dim_{\rm H} \mathcal{U}(\omega,\ell)\ge\dim_{\rm H}  \left\{y\in X\colon \overline{d}_\mu(y)<\frac{1}{\alpha}\right\}.\]
Note that
\[\left\{y\in X\colon \overline{d}_\mu(y)<\frac{1}{\alpha}\right\}=\bigcup_{m=1}^\infty\bigcup_{N=1}^\infty\bigcap_{n=N}^\infty\left\{ y\in X\colon \mu\left(B\left(y,\frac{1}{n}\right)\right)\ge n^{-\frac{1}{\alpha}+\frac{1}{m}}\right\}.\]
Given $m$, $N$ and $n\ge N$, let $(y_i)_{i\ge1}$ be a sequence of points of $\{ y\in X\colon \mu(B(y,\frac{1}{n}))\ge n^{-\frac{1}{\alpha}+\frac{1}{m}}\}=:A_{m,N,n}$ which converges to $y_0$. It follows from Fatou's lemma that
\begin{align*}
\mu\left(B\left(y_0,\frac{1}{n}\right)\right)\ge \limsup_{i\to\infty}\mu\left(B\left(y_i,\frac{1}{n}\right)\right)\ge n^{-\frac{1}{\alpha}+\frac{1}{m}}.
\end{align*}
Thus $y_0\in A_{m,N,n}$, implying that $A_{m,N,n}$ is closed. Then $\{y\in X\colon \overline{d}_\mu(y)<\frac{1}{\alpha}\}$ is analytic. For $0<s<\dim_{\rm H}\{y\in X\colon \overline{d}_\mu(y)<\frac{1}{\alpha}\}$, there exists a Borel probability measure $\nu$ and $c>0$ such that
\begin{itemize}
\item $\nu\left( \{y\in X\colon \overline{d}_\mu(y)<\frac{1}{\alpha}\}\right)=1.$
\item For $y\in X$ and $r>0$,
$\mu(B(y,r))\le cr^s.$
\end{itemize}
Applying Fubini theorem and Lemma \ref{equality}, we obtain that
\begin{align*}
&\int \nu\left( \left\{y\in X\colon \overline{H}(\omega,y)<\frac{1}{\alpha}\right\}\right)d \mathbb{P}\\
=&\int_{\{y\in X\colon \overline{d}_\mu(y)<\frac{1}{\alpha}\}} \mathbb{P}\left( \left\{\overline{H}(\omega,y)<\frac{1}{\alpha}\right\}\right)d \nu(y)\\
=&\int_{\{y\in X\colon \overline{d}_\mu(y)<\frac{1}{\alpha}\}} \mathbb{P}\left( \left\{\overline{H}(\omega,y)<\frac{1}{\alpha}\right\}\cap \{\overline{H}(\omega,y)=\overline{d}_\mu(y)\}\right)d \nu(y)\\
=&\int_{\{y\in X\colon \overline{d}_\mu(y)<\frac{1}{\alpha}\}} \mathbb{P}\left( \{\overline{H}(\omega,y)=\overline{d}_\mu(y)\}\right)d \nu(y)\\
=&\nu\left(\{y\in X\colon \overline{d}_\mu(y)<\frac{1}{\alpha}\}\right)=1,
\end{align*}
giving that almost surely we have $\nu\left( \left\{y\in X\colon \overline{H}(\omega,y)<\frac{1}{\alpha}\right\}\right)=1$. 
Therefore $\dim_{\rm H}\left\{y\in X\colon \overline{H}(\omega,y)<\frac{1}{\alpha}\right\}\ge s$. By Lemma \ref{subset}, we get
\[\dim_{\rm H}\mathcal{U}(\omega,\ell)\ge \dim_{\rm H}\left\{y\in X\colon \overline{H}(\omega,y)<\frac{1}{\alpha}\right\}\ge s,\]
then 
\[\dim_{\rm H}\mathcal{U}(\omega,\ell)\ge \dim_{\rm H}\left\{y\in X\colon \overline{d}_\mu(y)<\frac{1}{\alpha}\right\}.\]

\section{Application to the $d$-torus}\label{appli}
Let $d$ be an integer such that $d\geq1$. In this section, we apply the relevant theorems  in metric spaces to uniform random covering sets in the $d$-dimensional torus, denoted by $\mathbb{T}^d = \mathbb{R}^d / \mathbb{Z}^d$. We begin with  the  framework.
%设 $d\geq1$ 是整数. 本节将度量空间上一致随机覆盖集的相关定理应用到高维环面$\mathbb{T}^d=\mathbb{R}^d/\mathbb{Z}^d$上的一致随机覆盖集, 其框架如下:
%\begin{itemize}
  %  \item 
  Let $X=\mathbb{T}^d $. %Note that for any $x+ \mathbb{Z}^d\in \mathbb{R}^d/ \mathbb{Z}^d$, there exists a unique $x'\in[0,1)^d$ such that $x'\in x+ \mathbb{Z}^d$. Therefore, we will treat $ \mathbb{T}^d$ as $[0,1)^d$ in the following. 
  %易知对任意 $x+\Z^d\in\R^d/\Z^d$, 存在唯一的 $x'\in[0,1)^d$ 使得 $x'\in x+\Z^d$. 因此以下将 $\T^d$ 视作 $[0,1)^d$. 
For $x\in \mathbb{T}^d$, let
    \[
    \|x\|:=\min\{|x-z|:z\in \mathbb{Z}^d\}\,,
    \]
    where $|\cdot|$ denotes the maximum norm on $\mathbb{R}^d$. Let $\lambda$ be the Lebesgue measure restricted on $\mathbb{T}^d$.%$\ell^{\infty}$ norm on $\mathbb{R}^d$. 高维环面 $\T^d$ 上的度量取作由 $\|\cdot\|$ 诱导的度量.  
    
 %   \item $\lambda$ 是 $\mathbb{T}^d$ 上的Lebesgue测度.
    %\item $\forall \,r>0,\ B_{\mathbb{T}^d}(\cdot,r)$为高维环面上的球．
  %  \item 概率空间 $(\Omega,\cB,\PP)$ 的定义如下:
  %  \begin{itemize}
   %     \item[$\circ$] $\displaystyle\Omega:=(\mathbb{T}^d)^{\N}$ 是乘积空间;
   %     \item[$\circ$] $\displaystyle\PP:=\prod_{i=1}^{\infty}\lambda$ 是乘积测度;
    %    \item[$\circ$] $\cB$ 是 $\PP$-可测集构成的 $\sigma$-代数.
    %\end{itemize}
    Let $\omega=(\omega_n)_{n\geq 1}$ be an i.i.d. random sequence of uniform distribution on the torus $\mathbb{T}^d $.
Given $\ell=(\ell_n)_{n\geq 1}\subseteq [0,+\infty)$, uniform random covering sets are defined by
    \[
    \mathcal U(\omega,\ell):=\big\{y\in \mathbb{T}^d:\forall\, N\gg 1,\ \exists\, 1\leq n\leq N,\,\,\,\text{\text{s.t.}}\,\ \Vert \omega_{n}-y\Vert<\ell_N\big\}\,,
    \]
  where $\omega_n=(\omega_{n,1},\omega_{n,2},\,\cdots,\,\omega_{n,d})\in\mathbb{T}^d$.
%    \item 对任意 $k\geq1$, 定义随机变量 $X_k(\omega):=\omega_k$, 那么容易验证 $X_1,X_2,\cdots$ 独立同分布.
%\end{itemize}
%\subsection{0-1 law}
%首先, 应用定理~\ref{thm31}~到高维环面上$:$
%\renewcommand{\thecorollary}{3.1}
\begin{corollary}
Let $(\ell_n)_{n\geq1}$ be a decreasing sequence of positive real numbers. Let $0 \leq s_1 \leq d,$ $0\leq s_2\leq 1$. Then the following events 
  \begin{eqnarray*}
      A&:=&\left\{\mathcal{U}(\omega,\ell) =\mathbb{T}^d\right\},\\
      B&:=&\left\{\mathrm{dim}_\mathrm{H} \ \mathcal{U}(\omega,\ell) = s_1\right\},\\
      C&:=&\left\{\lambda\big(\mathcal{U}(\omega,\ell)\big)=s_2\right\}
  \end{eqnarray*}
    obey $0$-$1\ \text{law},$  that is,
  \[
  \mathbb{P}(A),\,\mathbb{P}(B),\,\mathbb{P}(C)\in\{0,1\}\,.
  \]
\end{corollary}
%\subsection{一致随机覆盖集的测度}
%由于 $\lambda$ 是 $\T^d$ 上的 Lebesgue 测度, 故对任意 $y\in\T^d$ 与任意 $r\in[0,1/2]$, 都有 $\lambda(   B_{\mathbb{T}^d}(y,r))=2^dr^d$.  根据定理~\ref{thm32}~以及定理~\ref{thm03}~可得如下推论.
%Since $\lambda$ is the Lebesgue measure on $\mathbb{T}^d$, 
For any $y\in\mathbb{T}^d$ and any $r\in[0,1/2]$, we have $\lambda(B(y,r))=2^dr^d$.  By Theorem~\ref{theorem2}, we obtain the following corollary.
\begin{corollary}
%假定 $\{r_n\}_{n\geq1}\ \text{递减},$ $\{n(r_n)^d\}_{n\geq1}$ 递增.  那么 $\lambda\left({\cal U}(\omega,\{r_n\})\right)$ 几乎必然等于 $1$ 当且仅当序列$\{r_n\}_{n\geq1}$ 满足
Let $(\ell_n)_{n\geq1}$ be a decreasing sequence of positive real numbers, and $(n\ell_n^d)_{n\geq1}$ be increasing. Then $\lambda\left({\mathcal U}(\omega,\ell)\right)=1$ almost surely if and only if 
$$\sum_{n=1}^{\infty}\ell_{n}^d=\infty\quad\text{and}\quad\sum_{n=1}^{\infty}\ell_{n}e^{-n(2\ell_{n})^d}<\infty\,.$$
\end{corollary}

%\subsection{一致随机覆盖集可数的充分条件}
%由定理~\ref{thm35}~可知, 若序列 $\{r_n\}_{n\geq1}$ 单调递减趋于零且$$\sum\limits_{n=1}^{\infty}n\cdot\int_{\mathbb{T}^d}\lambda\big(B(y,r_n+r_{n+1})\big)\,\rd\lambda(y)=\sum\limits_{n=1}^{\infty}n\big(2(r_{n}+r_{n+1})\big)^d<\infty,$$ 则几乎必然有${\cal U}(\omega,\{r_n\})=\big\{\omega_{k}:k\in\mathbb{N}\big\}.$于是有
By Theorem~\ref{theorem3}, if the sequence $(\ell_n)_{n\geq1}$ is monotonically decreasing  to zero, and $$\sum\limits_{n=1}^{\infty}n\cdot\int_{\mathbb{T}^d}\lambda\big(B(y,\ell_n+\ell_{n+1})\big)\,\ d\lambda(y)=\sum\limits_{n=1}^{\infty}n\big(2(\ell_{n}+\ell_{n+1})\big)^d<\infty,$$ then it  follows that ${\mathcal U}(\omega,\ell)=\big\{\omega_{k}:k\in\mathbb{N}\big\}$ almost surely (a.s. for short).  Thus, we have the following corollary.
\begin{corollary}
Let $(\ell_n)_{n\geq1}$ be a decreasing sequence of positive real numbers with $\sum\limits_{n=1}^{\infty}n\ell_{n}^d<\infty,$ Then almost surely we have 
%  若序列 $\{r_n\}_{n\geq1}$ 单调递减趋于零且$\sum\limits_{n=1}^{\infty}n(r_{n})^d<\infty,$ 则几乎必然有
$${\mathcal U}(\omega,\ell)=\big\{\omega_{k}:k\in\mathbb{N}\big\}.$$
\end{corollary}

%\subsection{一致随机覆盖集的 Hausdorff 维数}

%根据定义, 得到高维环面上 Lebesgue 测度的局部维数为
%$${d}_{\lambda}(y)=\lim\limits_{r\rightarrow0}\frac{\log\mu\big(B(y,r)\big)}{\log r}=\lim\limits_{r\rightarrow0}\frac{\log r^d}{\log r}=d\,.$$
%那么由定理~\ref{thm33}~可知
%By definition, the local dimension of the Lebesgue measure on $\mathbb{T}^d$ is given by ${d}_{\lambda}(y)=\lim\limits_{r\rightarrow0}\frac{\log\mu\big(B(y,r)\big)}{\log r}=\lim\limits_{r\rightarrow0}\frac{\log r^d}{\log r}=d\,.$ 
By definition, ${d}_{\lambda}(y)=d\,.$ Then the following corollary follows from Theorem~\ref{theorem4}.
\begin{corollary}
Let $\ell_n=n^{-\alpha}(\alpha>0),$  then almost surely
$$\dim_{\mathrm{H}}\mathcal{U}\left(\omega,(n^{-\alpha})\right)=
  \begin{cases}
    0,& \alpha>\frac{1}{d}, \\
   \\d,& \alpha<\frac{1}{d}.\\
  \end{cases} $$
\end{corollary}
%给定$\alpha>0,$ 假定$r_n=n^{-\alpha}.$ 当$\frac{1}{\alpha}<d$时, 根据定理~\ref{thm33}~有
%$$\dim_{\mathrm{H}}\mathcal{U}(\omega,\{n^{-\alpha}\})\leq \dim_{\mathrm{H}}\left\{y\in \mathbb{T}^d:\underline{d}_{\lambda}(y)\leq \frac{1}{\alpha}\right\}\leq \dim_{\mathrm{H}}\left\{y\in \mathbb{T}^d:\underline{d}_{\lambda}(y)\leq d\right\}=0\,,$$
%即$\dim_{\mathrm{H}}\mathcal{U}(\omega,\{n^{-\alpha}\})=0\,.$

%当$\frac{1}{\alpha}>d$时, 根据定理~\ref{thm33}~有
%$$\dim_{\mathrm{H}}\mathcal{U}(\omega,\{r_n\})\geq\dim_{\mathrm{H}}\left\{y\in \mathbb{T}^d:\overline{d}_{\lambda}(y)< \frac{1}{\alpha}\right\}=d\,,$$
%即$\dim_{\mathrm{H}}\mathcal{U}(\omega,\{n^{-\alpha}\})=d\,.$

In particular, when $\alpha=\frac{1}{d}$, by Theorem ~\ref{theorem4}, we get that 
\begin{multline*}
0=\dim_{\mathrm{H}}\left\{y\in \mathbb{T}^d:\underline{d}_{\lambda}(y)< \frac{1}{\alpha}\right\}\leq\dim_{\mathrm{H}}\mathcal{U}\left(\omega,(n^{-\alpha})\right)\leq\dim_{\mathrm{H}}\left\{y\in \mathbb{T}^d:\overline{d}_{\lambda}(y)\leq \frac{1}{\alpha}\right\}=d\,.
\end{multline*}
%In this case, we obtain nothing on the Hausdorff dimension of  $\mathcal U\left(\omega,(n^{-1/d})\right)$. Therefore, a more detailed discussion is needed for this case. When $d=1$, Koivusalo et al. \cite{ KLP} investigated the Hausdorff dimension of the uniform random covering set $\mathcal{U}(\omega,(\frac{c}{n}))$, and in the following, we extend their results to higher dimensional case, more precisely, we consider the case where $d\ge2$ and $\ell_n=\frac{c}{n^{1/d}}$ for $c>0$, $n\in\mathbb N$. 
In this case, the Hausdorff dimension of $\mathcal{U}\left(\omega, (n^{-1/d})\right)$ remains undetermined. A more detailed analysis is therefore required. For the case $d = 1$, Koivusalo et al. \cite{KLP} studied the Hausdorff dimension of the uniform random covering set $\mathcal{U}\left(\omega, \left(\frac{c}{n}\right)\right)$. In this paper, we extend their results to higher dimensions—specifically, to the case where $d \geq 1$ and $\ell_n = \frac{c}{n^{1/d}}$ for some $c > 0$ and $n \in \mathbb{N}$.

\begin{theorem}\label{thm41}
Let $\ell_n=\frac{c}{n^{1/d}},$ $n \in \mathbb{N}$, then
\begin{itemize}
  \item[$(1)$] $\text{if}\ c>\frac{1}{2}\ $, we get
\[ \mathrm{dim}_{\mathrm{H}}\,\mathcal{U}\left(\omega,\left(\frac{c}{n^{1/d}}\right)\right)\geq\sup\limits_{\theta>1}d\Big(1+\frac{\log\big(1-e^{(-(2c)^d\frac{\theta-1}{\theta^2})}\big)}{\log\theta}\Big)\quad a.s.,\]
  \item[$(2)$] $\text{if}\ 0<c< \frac{1}{2}\ $, we get
$$ \mathrm{dim}_{\mathrm{H}}\,\mathcal{U}\left(\omega,\left(\frac{c}{n^{1/d}}\right)\right)\leq\inf\limits_{\theta>1}\frac{\log\Lambda^d}{\log\theta}\quad a.s.,$$
where
\begin{eqnarray*}
\Lambda&=&\Lambda(\theta)\ :=\ \frac{1+\Theta+\Delta}{2}+\sqrt{\Big(\frac{1+\Theta+\Delta}{2}\Big)^{2}-\Delta}\,,\\
    \Theta&=&\Theta(\theta)\ :=\ (\theta-1)\big(2c(1+\theta^{-\frac{2}{d}})\big)^d\,,\\
    \Delta&=&\Delta(\theta)\ :=\ (\theta-1)2^dc^d\big((1+\theta^{-\frac{1}{d}})^d-(1+\theta^{-\frac{2}{d}})^d\big)\,.
\end{eqnarray*}
\end{itemize}
\end{theorem}

\begin{remark}\label{re41}
\begin{itemize}
    \item[$(1)$] For the lower bound on the $\mathrm{Hausdorff}$ dimension:
    $$ \text{If}\ c>\frac{1}{2}, \ \text{then}\ \sup\limits_{\theta>1}\Big(d+\frac{d\log\big(1-e^{(-(2c)^d\frac{\theta-1}{\theta^2})}\big)}{\log\theta}\Big)>0.$$
    $$\text{If}\ 0<c\leq\frac{1}{2},\ \text{then}\  \sup\limits_{\theta>1}\Big(d+\frac{d\log\big(1-e^{(-(2c)^d\frac{\theta-1}{\theta^2})}\big)}{\log\theta}\Big)=0.$$
    \item[$(2)$] For the upper bound on the $\mathrm{Hausdorff}$ dimension$:$
  $$\text{If}\ 0<c<\frac{1}{2},\ \text{then}\  0<\inf\limits_{\theta>1}\frac{\log\Lambda^d}{\log\theta}<d.$$
  $$\text{If}\ c\geq\frac{1}{2},\ \text{then}\  \inf\limits_{\theta>1}\frac{\log\Lambda^d}{\log\theta}=d.$$
\end{itemize}
\end{remark}
The following corollary follows from Theorem~\ref{thm41}~and Remark~\ref{re41}.

\begin{corollary}
Let $\ell_n=\frac{c}{n^{1/d}},~n\in\mathbb{N}.$
\begin{itemize}
    \item[$(1)$] $\text{If}\ c>\frac{1}{2}, \ \text{then}\   \mathrm{dim}_{\mathrm{H}}\,\mathcal{U}\left(\omega,\left(\frac{c}{n^{1/d}}\right)\right)>0\quad a.s..$
\item[$(2)$]   $\text{If}\ 0<c< \frac{1}{2},\ \text{then}\  \mathrm{dim}_{\mathrm{H}}\,\mathcal{U}\left(\omega,\left(\frac{c}{n^{1/d}}\right)\right)<d\quad a.s..$ 
\end{itemize}
\end{corollary}

Now we give the proof of Theorem \ref{thm41},  which is divided into two parts: the upper bound and the lower bound.
\subsection{ The upper bound on Hausdorff dimension}
%\begin{proof}[{定理~{\normalfont\ref{thm41}}~中上界的证明}]
Fix $\theta>1$ and $l\in\mathbb N$. Put $n_j = \theta^j$ $(j\in\mathbb N)$. %首先, 对任意 $i\geq l$, 我们构造下列集合的覆盖$$G_{l,i}=\bigcap_{j=l}^{i}\bigcup_{k=1}^{n_{j}}B_{\mathbb{T}^d}\(\omega_k,r_{n_j}\).$$其构造思路如下: 寻找 $I_i,\,J_i\subset\{1,2,...,n_i\}$ 使得 $I_i\cap J_i=\emptyset$ 且
First, for any $i\geq l$, we construct a class of covering sets of the following set
$$G_{l,i}: =\bigcap_{j=l}^{i}\bigcup_{k=1}^{n_{j}}B(\omega_k,\ell_{n_j}).$$
The construction  is as follows: choose $I_i,\,J_i\subset\{1,2,...,n_i\}$ such that $I_i\cap J_i=\emptyset$ and 
\begin{equation}\label{coveriiji}
    G_{l,i}\subset\bigcup_{h\in I_{i}\cup J_{i}}B(\omega_h, \ell_{n_i})=: H_i.
\end{equation}
%令 $N_i:=\# I_i$, $Q_i:=\# J_i$, 易见该覆盖由 $N_i+Q_i$ 个球 $B(\omega_h,r_{n_i})$~$(h\in I_i\cup J_i)$ 组成. 接下来归纳地构造 $I_i$ 与 $J_i$:
Denote $N_i:=\# I_i$, $Q_i:=\# J_i$. Then the covering sets consist of $N_i+Q_i$ balls $B(\omega_h,\ell_{n_i})$,~$h\in I_i\cup J_i$. Next, we construct $I_i$ and $J_i$ inductively.

\medskip

\noindent\emph{$\bullet$Step 1}: If $i=l$, put $I_l := \left\{1, 2, ... ,n_l\right\}$ and $J_l := \emptyset$. 
\medskip

\noindent\emph{$\bullet$Step 2}: %假定$I_i$和$J_i$已被定义. 那么 $I_{i+1}$ 与 $J_{i+1}$ 的定义如下.    首先定义 $I_{i+1}$:$$I_{i+1}:=I_{i}\cup T_{i+1},$$ 其中
Assume that $I_i$ and $J_i$ have been constructed. Then $I_{i+1}$ and $J_{i+1}$ are defined as follows.  Define $I_{i+1}$ as:
$$I_{i+1}:=I_{i}\cup T_{i+1},$$
where
%$$T_{i+1}:=\left\{n_{i}+1\leq k\leq n_{i+1}\left|\ \begin{aligned}
%    &B(\omega_k, \ell_{n_{i+1}})\cap H_i\not=\emptyset\ \text{~and~there~exists~}\  h\in I_i\cup J_i \\ 
%     &\text{such ~that}\ \Vert \omega_{k}-\omega_{h}\Vert<\ell_{n_i}+\ell_{n_{i+2}}
%\end{aligned}\right.\right\}\ .$$
\begin{multline*}
T_{i+1}:=\big\{n_{i}+1\leq k\leq n_{i+1} \colon
    B(\omega_k, \ell_{n_{i+1}})\cap H_i\not=\emptyset\ \text{~and~there}\\
  \text{exists~}    h\in I_i\cup J_i  \ \text{such ~that}\ \Vert \omega_{k}-\omega_{h}\Vert<\ell_{n_i}+\ell_{n_{i+2}}\big\}.
\end{multline*}

Next we define $J_{i+1}$ as follows.
\begin{multline*}
J_{i+1}:=\big\{n_{i}+1\leq k\leq n_{i+1} \colon
    B(\omega_k, \ell_{n_{i+1}})\cap H_i\not=\emptyset\ \text{~and~for}\\ 
 \text{any~}\     h\in I_i\cup J_i\ \text{we~have~}\ \Vert \omega_{k}-\omega_{h}\Vert\geq \ell_{n_i}+\ell_{n_{i+2}}\big\}.
\end{multline*}
Note that  $I_{i+1}\cap J_{i+1}=\emptyset$.
\medskip

%接下来验证 \eqref{coveriiji} 式对上述构造的 $\{I_i\}_{i\geq l}$与$\{J_i\}_{i\geq l}$ 成立. 实际上, 我们将用数学归纳法证明以下陈述为真:
In the following, we verify that equation \eqref{coveriiji} hold for  $\{I_i\}_{i\geq l}$ and $\{J_i\}_{i\geq l}$. % In fact, we will use mathematical induction to prove that the following statement is true.
We claim that 
 \begin{itemize}
    \item[(a)] For any $r\leq \ell_{n_{i+1}}$ and integer $i\geq l$, we have
    \begin{eqnarray*}
        G_{l,i}\cap\left(\bigcup_{k=1}^{n_i}B(\omega_k,r)\right)\subseteq\bigcup_{k\in I_i}B(\omega_k,r)\,.
    \end{eqnarray*}

    \item[(b)] For any integer  $i\geq l$, \eqref{coveriiji} holds.
\end{itemize}
Now we will use mathematical induction to prove that the claim above is true.

\noindent\emph{$\bullet$Step 1:} The above statements (a) and (b) clearly hold for $i=l$.
\medskip

\noindent\emph{$\bullet$Step 2:} Assume that (a) and (b)  hold for $i$, and then we shall prove that these two statements also hold for $i+1$.

\noindent  1).  Let $r\leq \ell_{n_{i+2}}$, then
\begin{multline}\label{glicapcupbr}
 G_{l,i+1}\cap\left(\bigcup_{k=1}^{n_{i+1}}B(\omega_k,r)\right)=G_{l,i}\cap\left(\bigcup_{k=1}^{n_i}B(\omega_k,\ell_{n_{i+1}})\right)\\
 \cap\left(\bigcup_{k=1}^{n_i}B(\omega_k,r)\cup\bigcup_{k=n_i+1}^{n_{i+1}}B(\omega_k,r)\right). 
 \end{multline}
%    \begin{eqnarray}
%&~&\hspace{-16ex} G_{l,i+1}\cap\left(\bigcup_{k=1}^{n_{i+1}}B(\omega_k,r)\right)=G_{l,i}\cap\left(\bigcup_{k=1}^{n_i}B(\omega_k,\ell_{n_{i+1}})\right)\nonumber\\
 %       &~&\hspace{16.5ex} \cap\left(\bigcup_{k=1}^{n_i}B(\omega_k,r)\cup\bigcup_{k=n_i+1}^{n_{i+1}}B(\omega_k,r)\right). \label{glicapcupbr}
%    \end{eqnarray}
Since statement (a) holds for $i$, therefore
\begin{eqnarray}\label{glisubbtdwkr}
    G_{l,i}\cap\left(\bigcup_{k=1}^{n_i}B(\omega_k,r)\right)\subseteq\bigcup_{k\in I_i}B(\omega_k,r)\,.
\end{eqnarray}
%另一方面, 由于 (2) 对 $i$ 成立, 故 $G_{l,i}\subseteq H_i$. 根据 $T_{i+1}$ 的定义与 $r$ 的范围有 
On the other hand, since (a) holds for $i$, we have $G_{l,i}\subseteq H_i$. According to the definition of $T_{i+1}$ and the range of $r$, we have
\begin{eqnarray}\label{btdcapempty}
    B(\omega_k,r)\cap H_i=\emptyset,\quad\forall\,k\in \{n_i+1,...,n_{i+1}\}\setminus T_{i+1}.
\end{eqnarray}
%将 \eqref{glisubbtdwkr} 式与 \eqref{btdcapempty} 式代入  \eqref{glicapcupbr} 式可得
Substituting equations \eqref{glisubbtdwkr} and \eqref{btdcapempty} into equation \eqref{glicapcupbr} yields that
\begin{eqnarray*}
    G_{l,i+1}\cap\left(\bigcup_{k=1}^{n_{i+1}}B(\omega_k,r)\right)\subseteq\bigcup_{k\in I_i\cup T_{i+1}}B(\omega_k,r)=\bigcup_{k\in I_{i+1}}B(\omega_k,r)\,.
\end{eqnarray*}
This proves that (a) holds for $i+1$.

\noindent  2). Since (a) holds for $i$ and according to the definitions of $I_{i+1}$, $T_{i+1}$ and $J_{i+1}$, we obtain that %由于 (1) 对 $i$ 成立且根据 $I_{i+1}$, $T_{i+1}$ 与 $J_{i+1}$ 的定义可知
\begin{eqnarray*}
    G_{l,i+1}&=&G_{l,i}\cap\left(\bigcup_{k=1}^{n_{i+1}}B(\omega_k,\ell_{n_{i+1}})\right)\\
    &=&\left(G_{l,i}\cap\Big(\bigcup_{k=1}^{n_{i}}B(\omega_k,\ell_{n_{i+1}})\Big)\right)\cup\left(G_{l,i}\cap\Big(\bigcup_{k=n_i+1}^{n_{i+1}}B(\omega_k,\ell_{n_{i+1}})\Big)\right)\\
    &\subseteq&\bigcup_{k\in I_i}B(\omega_k,\ell_{n_{i+1}})\ \ \cup\bigcup_{k\in T_{i+1}\cup J_{i+1}}B(\omega_k,\ell_{n_{i+1}})\\[4pt]
    &=&H_{i+1}\,.
\end{eqnarray*}
Therefore (b) holds for $i+1$.%因此 (2) 对 $i+1$ 成立.

\medskip

%下面建立二元组 $\big(\mathbb{E}(N_{i+1}),\mathbb{E}(Q_{i+1})\big)$ 与 $\big(\mathbb{E}(N_i),\mathbb{E}(Q_i)\big)$ 之间的不等式. 根据 $\{N_i\}_{i\geq l}$ 与 $\{Q_i\}_{i\geq l}$ 的定义可知, 对任意 $i\geq l$ 与 $\omega\in\Omega$,
Next we establish the relationship between the pairs $\big(\mathbb{E}(N_{i+1}),\mathbb{E}(Q_{i+1})\big)$ and $\big(\mathbb{E}(N_i),\mathbb{E}(Q_i)\big)$. According to the definitions of $(N_i)_{i\geq l}$ and $(Q_i)_{i\geq l}$, for any $i\geq l$ and $\omega$, we have
\begin{eqnarray}\label{ni1leni}
  %  \left\{\right.
    \begin{array}
  {rcl}N_{i+1}&\leq& N_{i}+\sum\limits_{k={n_{i}+1}}^{n_{i+1}}\mathbbm{1}_{\mathcal{T}_{i+1}}(\omega_{k})\,,\\
Q_{i+1}&\leq&\sum\limits_{k={n_{i}+1}}^{n_{i+1}}\mathbbm{1}_{\mathcal{J}_{i+1}}(\omega_{k})\,,
\end{array}
\end{eqnarray}
where
\begin{eqnarray*}
    \begin{aligned}
        &\mathcal{T}_{i+1}=\mathcal{T}_{i+1}(\omega_1,...,\omega_{n_i}):=\bigcup_{h\in I_i\cup J_i}B(\omega_h,\ell_{n_i}+\ell_{n_{i+1}}),\\
        &\mathcal{J}_{i+1}=\mathcal{J}_{i+1}(\omega_1,...,\omega_{n_i}):=\bigcup_{h\in I_i\cup J_i}B(\omega_h,\ell_{n_i}+\ell_{n_{i+1}})\setminus B(\omega_h,\ell_{n_i}+\ell_{n_{i+2}}).
    \end{aligned}
\end{eqnarray*}
%设$\mathscr{S}_i:=\sigma(X_1,X_{2},\cdots,X_{n_i}),$那么由 \eqref{ni1leni} 式可以推出
 Let $\mathscr{S}_i:=\sigma(X_1,X_{2},\cdots,X_{n_i}),$ then from equation \eqref{ni1leni} we can deduce that
$$%\left\{\right.
\begin{array}
  {rcl}\mathbb{E}(N_{i+1}|\mathscr{S}_i)&\leq&N_{i}+\big(2(\ell_{n_{i}}+\ell_{n_{i+2}})\big)^d(n_{i+1}-n_{i})(N_{i}+Q_{i}),\\\mathbb{E}(Q_{i+1}|\mathscr{S}_i)&\leq&2^d\big((\ell_{n_i}+\ell_{n_{i+1}})^d-(\ell_{n_i}+\ell_{n_{i+2}})^d\big)(n_{i+1}-n_{i})(N_{i}+Q_{i}).
\end{array}$$
Then
$$%\left\{\right
\begin{array}{rcl}\mathbb E(N_{i+1})&\leq&\mathbb E(N_{i})+
2^d(r_{n_{i}}+r_{n_{i+2}})^d(n_{i+1}-n_{i}) \ (\mathbb E(N_{i})+\mathbb E(Q_{i}) \ ),\\\mathbb E(Q_{i+1})&\le&2^d\big((\ell_{n_i}+\ell_{n_{i+1}})^d-(\ell_{n_i}+\ell_{n_{i+2}})^d\big)(n_{i+1}-n_{i})\ (\mathbb E(N_{i})+ \mathbb{E}(Q_{i})\ ).
\end{array}$$

\noindent Let
$$%\left\{\right.
\begin{array}
  {rcl}\Theta&:=&(\theta-1)\big(2c(1+\theta^{-\frac{2}{d}})\big)^d,\\\Delta&:=&(\theta-1)2^dc^d\big((1+\theta^{-\frac{1}{d}})^d-(1+\theta^{-\frac{2}{d}})^d\big).
\end{array}$$

\noindent Then

$$\begin{bmatrix}\mathbb E(N_{i+1})\\\mathbb E(Q_{i+1})\end{bmatrix}\leq\begin{bmatrix}1+\Theta&\Theta\\\Delta&\Delta\end{bmatrix}\begin{bmatrix}\mathbb E(N_i)\\\mathbb E(Q_i)\end{bmatrix}.$$

\noindent The maximum eigenvalue  of the above matrix is 

$$\Lambda=\frac{1+\Theta+\Delta}{2}+\sqrt{\Big(\frac{1+\Theta+\Delta}{2}\Big)^{2}-\Delta},$$
%并且存在常数 $C_0$使得 $\mathbb{E}(N_{i})+\mathbb{E}(Q_{i})\leq C_{0}\Lambda^{i}$. 对于任意的$\epsilon>0$, 几乎必然地存在常数$C$使得
and there exists a constant $C_0$ such that $\mathbb{E}(N_{i})+\mathbb{E}(Q_{i})\leq C_{0}\Lambda^{i-l}$. For any $\varepsilon>0$, almost surely there  exists a constant $C$ such that
$$N_{i}+Q_{i}\leq C(1+\varepsilon)^{i-l}\Lambda^{i-l}.$$

Note that %下面给出集合
$$G_{l}:=\bigcap_{j={l}}^{\infty}\bigcup_{k=1}^{n_{j}}B(\omega_k, \ell_{n_{j}})$$
is covered by $N_i + Q_i$ balls $B(\omega_k, \ell_{n_{i+1}})$. Then
\begin{align*}\label{be2}
\mathrm{dim}_{\mathrm{H}}G_{l}\leq\underline{\mathrm{dim}}_{\mathrm{B}}G_{l}\leq\liminf\limits_{i\rightarrow\infty}\frac{\log (N_i+Q_i)}{-\log \ell_{n_i}}\leq\frac{d\big(\log(1+\varepsilon)+\log\Lambda\big)}{\log\theta}.
\end{align*}

Since $\varepsilon$ is arbitrary, we obtain that $$\mathrm{dim}_{\mathrm{H}}G_{l}\leq\frac{\log\Lambda^d}{\log\theta}.$$
Since $\mathcal{U}(\omega,\ell)\subset\bigcup_l \,G_{l},$ combined with the countable stability of the Hausdorff dimension, it almost surely follows that $$\mathrm{dim}_{\mathrm{H}}\,\mathcal{U}(\omega,\ell)\leq\inf\limits_{\theta>1}\frac{\log\Lambda^d}{\log\theta}.$$
%由于 $\epsilon$ 可以足够小, 在~\eqref{be2}~式中得到 $$\mathrm{dim}_{\mathrm{H}}G_{l}\leq\frac{\log\Lambda^d}{\log\theta}.$$由于$\mathcal{U}(\omega,\{r_n\})\subset\bigcup \,G_{l},$结合 Hausdorff 维数的可数稳定性几乎必然有$$\mathrm{dim}_{\mathrm{H}}\,\mathcal{U}(\omega,\{r_n\})\leq\inf\limits_{\theta>1}\frac{\log\Lambda^d}{\log\theta}.$$
%\end{proof}

 In the following, we consider the lower bound of the Hausdorff dimension of uniform random covering sets. 
\subsection{ The lower bound on Hausdorff dimension of $\mathcal{U}(\omega,\ell)$}
%假定$\{B_{\mathbb{T}^d}(\omega_n,r_n)\}_{n\ge1}$ 为一列闭球. 设 $\{n_j\}_{j\geq 1}$ 为一列严格单调增的整数列. 
Assume that  $\{B(\omega_n,\ell_n)\}_{n\ge1}$ is a sequence of closed balls, which does not change the Hausdorff dimension of $\mathcal{U}(\omega,\ell)$. Let $(n_j)_{j\geq 1}$ be a strictly increasing sequence of integers. Put
$$F_{j}:=\bigcup\limits_{k=n_{j-1}+1}^{n_{j}}B(\omega_k, \ell_{n_{j+1}}).$$
Then $F_j$ is a compact set.
Since for each $n$ with $n_j<n\leq n_{j+1},$ we have $F_{j}\subset\bigcup\limits_{k=1}^{n}B(\omega_k, r_{n}).$
Therefore, 
% and we can get that 
\begin{equation}\label{define:F}
F:=\liminf\limits_ j F_{j}\subset \mathcal{U}(\omega,\ell).
\end{equation}
%Since for each $n$ with $n_j<n\leq n_{j+1},$ we have $F_{j}\subset\bigcup\limits_{k=1}^{n}B_{\mathbb{T}^d}(\omega_k, r_{n}).$Therefore, when $n_{j_0}<q$, 
%$$\bigcap\limits_{j=j_0}^{\infty}F_{j}\subset\bigcap_{n=q}^{\infty}\bigcup_{k=1}^{n}B_{\mathbb{T}^d}(\omega_k, r_{n}).$$
%定义 $\mathbb{T}^d$上的测度 $\mu_{l,q}$ 为

Suppose that $\{U_i\}_{i\geq1}$ is an open covering of $F$. Then, for any $l$,  $\{U_i\}_{i\geq1}$ is also an open covering of $\bigcap\limits_{j= l}^{\infty}F_j$. Since the sets $F_j$ are compact, there exists a $q$ such that $\{U_i\}_{i\geq1}$ covers $\bigcap\limits_{j= l}^qF_j$.

Define the measure $\mu_{l,q}$ on $\mathbb{T}^d$ as
$$\frac{\mathrm{d}\mu_{l,q}}{\mathrm{d}x}=\prod_{j=l}^{q}\mathbbm{1}_{F_{j}},$$
%因此$\mu_{l,q}$在$\bigcap\limits_{j= l}^qF_j$上有支撑.
%事实上, $\mu _{l,q}(A) = \int _{A}\prod\limits_{j= l}^q\mathbbm{1}_{F_j}(x)\rd x=\int_{A}\mathbbm{1}_{\bigcap\limits_{j=l}^qF_j}(x)\rd x,$即$\mu_{l,q}$是限制在$\bigcap\limits_{j= l}^qF_j$上的Lebesgue测度.
Therefore, $\mu_{l,q}$ is supportted on $\bigcap\limits_{j= l}^qF_j$.
In fact, %$\mu _{l,q}(A) = \int _{A}\prod\limits_{j= l}^q\mathbbm{1}_{F_j}(x) \mathrm{d}x=\int_{A}\ \mathbbm{1}_{\bigcap\limits_{j=l}^qF_j}(x)\mathrm{d} x,$ i.e., 
$\mu _{l,q}(A) = \int _{A}\prod\limits_{j= l}^q\mathbbm{1}_{F_j}(x) \mathrm{d}x=\lambda\bigl(A\cap \bigcap\limits_{j=l}^qF_j\bigr)$, that is, $\mu_{l,q}$ is the Lebesgue measure restricted to $\bigcap\limits_{j= l}^qF_j$.
   
%假设$\{U_i\}_{i\geq1}$ 是$F$的开覆盖. 那么对于任意$l$, 序列$\{U_i\}_{i\geq1}$是$\bigcap\limits_{j= l}^{\infty}F_j$的开覆盖. 因为集合$F_j$是紧集, 故存在$q$使得$\{U_i\}_{i\geq1}$覆盖$\bigcap\limits_{j= l}^qF_j$.

\renewcommand{\thelemma}{4.2}
\begin{lemma}\label{lem41} 
For $x,y\in \mathbb{T}^{d},$ we have
\begin{align*}
&~\hspace{5ex}\mathbb{E}\left(\mathbbm{1}_{B(\omega_k, \ell_{n})}(x)\mathbbm{1}_{B(\omega_k, \ell_{n})}(y)\right) \leq (2\ell_n)^d\mathbbm{1}_{B(0,2\ell_n)}(x-y),
\end{align*}
and 
\begin{align*}
&\mathbb E\left(\big(1-{\mathbbm{1}}_{B(\omega_k, \ell_{n})}(x)\big)\big(1-{\mathbbm{1}}_{B(\omega_k, \ell_{n})}(y)\big)\right)\\
&\hspace{5ex}~\hspace{5ex}~\hspace{5ex}\leq 1-2(2\ell_n)^d+(2\ell_n)^d\mathbbm{1}_{B(0,2\ell_n)}(x-y).
\end{align*}
%&\hspace{5ex}~\hspace{5ex}\hspace{5ex}=\int \mathbbm{1}_{B_{\mathbb{T}^d}(\omega_k, r_{n})}(x)\mathbbm{1}_{B_{\mathbb{T}^d}(\omega_k, r_{n})}(y)\mathrm{d}\mathbb{P}\\&\hspace{5ex}~\hspace{5ex}\hspace{5ex}\leq (2r_n)^d\mathbbm{1}_{B_{\mathbb{T}^d}(0,2r_n)}(x-y).
\end{lemma}
\begin{proof}
Let $x=(x_1,x_2,\dots,x_d)$, $y=(y_1,y_2,\dots,y_d) $, then 
\begin{align*}
\int \mathbbm{1}_{B(\omega_k, \ell_{n})}(x)\mathbbm{1}_{B(\omega_k, \ell_{n})}(y)\mathrm{d}\mathbb{P}&=\left\{\begin{array}{ll}0,&\text{if}\ \Vert x-y\Vert>2\ell_{n},\\
\prod\limits_{i=1}^d(2\ell_{n}-|x_i-y_i|),&\text{if}\ \Vert x-y\Vert\leq 2\ell_{n}.\end{array}\right.\\
&\leq \left\{\begin{array}{ll}0,&\text{if}\ \Vert x-y\Vert>2\ell_{n},\\ (2r_n)^d,&\text{if}\ \Vert x-y\Vert\leq 2\ell_{n}.\end{array}\right.
\end{align*}
Therefore
\[\mathbb{E}\left(\mathbbm{1}_{B(\omega_k, \ell_{n})}(x)\mathbbm{1}_{B(\omega_k, \ell_{n})}(y)\right)\leq (2\ell_n)^d\mathbbm{1}_{B(0,2\ell_n)}(x-y).\]
%\end{proof}

%\renewcommand{\thelemma}{4.2}
%\begin{lemma}\label{lem42}
%For $x,y\in \mathbb{T}^{d},$ we have
%\begin{align*}
%&\mathbb E\left[\big(1-{\mathbbm{1}}_{B_{\mathbb{T}^d}(\omega_k, r_{n})}(x)\big)\big(1-{\mathbbm{1}}_{B_{\mathbb{T}^d}(\omega_k, r_{n})}(y)\big)\right]\\
%&\hspace{5ex}~\hspace{5ex}~\hspace{5ex}\leq 1-2(2r_n)^d+(2r_n)^d\mathbbm{1}_{B_{\mathbb{T}^d}(0,2r_n)}(x-y).
%\end{align*}
%\end{lemma}

%\begin{proof}
Note that 
\begin{align*}
\mathbb E \left({\mathbbm{1}}_{B(\omega_k, \ell_{n})}(x)\right )=\int\mathbbm{1}_{B(\omega_k, \ell_{n})}(x)\mathrm{d}\mathbb{P}=(2\ell_{n})^{d}, 
%&=\mathbb{P}\big(\{{\omega:\Vert \omega_{k}-x\Vert\leq r_{n}}\}\big)\\
%&=\mathcal{L}((\mathbb{T}^m)^{\mathbb{N}})\times\lambda(d(w^{k},z)<r_{n})\\
\end{align*}
then
\begin{align*}
&\mathbb E\left((1-{\mathbbm{1}}_{B(\omega_k, \ell_{n})}(x))(1-{\mathbbm{1}}_{B(\omega_k, \ell_{n})}(y))\right)\\
&=1-\mathbb E\left(\mathbbm{1}_{B(\omega_k, \ell_{n})}(x)\right)-\mathbb E\left(\mathbbm{1}_{B(\omega_k, \ell_{n})}(y)\right)+\mathbb E\left[\mathbbm{1}_{B(\omega_k, \ell_{n})}(x)\mathbbm{1}_{B(\omega_k, \ell_{n})}(y)\right]\\
&= 1-2(2\ell_n)^d+\mathbb E\left[\mathbbm{1}_{B(\omega_k, \ell_{n})}(x)\mathbbm{1}_{B(\omega_k, \ell_{n})}(y)\right]\\
&\leq 1-2(2\ell_n)^d+(2\ell_n)^d\mathbbm{1}_{B(0,2\ell_n)}(x-y).
\end{align*}
\end{proof}

For $t\in \mathbb{T}^d$, put
 $$\Psi_{l,q}(t):=\prod\limits_{j=l}^q\left(1+\frac{\big(1-(2\ell_{n_{j+1}})^{d}\big)^{n_{j}-n_{j-1}}}{1-\big(1-(2\ell_{n_{j+1}})^{d}\big)^{n_{j}-n_{j-1}}}\mathbbm{1}_{B(0,\ell_{n_{j+1}})}(t)\right).$$

\begin{lemma}\label{lem43}
If $n_j=\theta^j$ and $\ell_n=\frac{c}{n^{1/d}}$, $c>0$,  then 
$$\Psi_{l,q}(t)\leq 1+C_l\Vert t\Vert^{-s(c,\theta)},$$
 where
$$C_l=c^{s(c,\theta)}\big(1-e^{-(2c)^d\frac{\theta-1}{\theta^2}}\big)^l$$and$$s(c,\theta)=\frac{-d\log\big(1-e^{-(2c)^d\frac{\theta-1}{\theta^2}}\big)}{\log\theta}.$$
\end{lemma}
\begin{proof}
%选取$\Vert t\Vert>0.$ 如果 $\Vert t\Vert>r_{n_{l+1}}$, 那么根据$\Psi_{l,q}(t)$的定义可知 $\Psi_{l,q}(t)=1$. 否则存在 $j_0 > l$ 使得当$r_{n_{{j_0}+1}} \leq \Vert t\Vert < r_{n_{j_0}}$时, 有
Take $t$ with $\Vert t\Vert>0.$ If $\Vert t\Vert>\ell_{n_{l+1}}$, then by the definition of $\Psi_{l,q}(t)$, we get that $\Psi_{l,q}(t)=1$. 

If $0<\Vert t\Vert\leq\ell_{n_{l+1}}$, there exists $j_0 > l$ such that  $\ell_{n_{{j_0}+1}} \leq \Vert t\Vert < \ell_{n_{j_0}}$, we have
\begin{align*}
\Psi_{l,q}(t)&\leq \prod_{j=l}^{j_{0}}\biggl(1+\frac{(1-(2\ell_{n_{j+1}})^{d})^{n_{j}-n_{j-1}}}{1-\big(1-(2\ell_{n_{j+1}})^{d}\big)^{n_{j}-n_{j-1}}}\biggr) \\
&=\prod_{j=l}^{j_{0}}\biggl(\frac{1}{1-\big(1-(2\ell_{n_{j+1}})^{d}\big)^{n_{j}-n_{j-1}}}\biggr).
\end{align*}

Since $n_j = \theta^j$ and $r_n=\frac{c}{n^{\frac{1}{d}}},$ we get that 
\begin{align*}
\Psi_{l,q}(t)&\leq \prod_{j=l}^{j_{0}}\biggl(\frac{1}{1-\big(1-(2c)^d\theta^{-(j+1)}\big)^{{\theta^{j+1}}(\theta-1)\theta^{-2}}}\biggr)\\
&\leq\prod_{j=l}^{j_{0}}\Bigl(\frac{1}{1-e^{-(2c)^d\frac{\theta-1}{\theta^{2}}}}\Bigr),
\end{align*}
where the last inequality holds since $x\mapsto(1-1/x)^{x}$ is increasing.

Furthermore, recall that $\Vert t\Vert <\ell_{n_{j_0}}=c\theta^{-j_0/d}$, then 
\begin{align*}
\Psi_{l,q}(t)&\leq \Bigl(1-e^{-(2c)^d\frac{\theta-1}{\theta^{2}}}\Bigr)^{-j_0+l}\\
&=\Bigl(1-e^{-(2c)^d\frac{\theta-1}{\theta^{2}}}\Bigr)^{l}\Bigl(1-e^{-(2c)^d\frac{\theta-1}{\theta^{2}}}\Bigr)^{-j_0}\\
&=\Bigl(1-e^{-(2c)^d\frac{\theta-1}{\theta^{2}}}\Bigr)^{l}(\theta^{-j_0/d})^{\frac{d\log (1-\exp\{-(2c)^d\frac{\theta-1}{\theta^{2}}\})}{\log\theta}}\\
&=C_{l}(c\theta^{-j_0/d})^{-s(c,\theta)}\leq C_{l}\Vert t\Vert^{-s(c,\theta)},
\end{align*}
here $$C_l=c^{s(c,\theta)}\big(1-e^{-(2c)^d\frac{\theta-1}{\theta^2}}\big)^l$$ and $$s(c,\theta)=\frac{-d\log\big(1-e^{-(2c)^d\frac{\theta-1}{\theta^2}}\big)}{\log\theta}.$$
%综上可知, 对任意$t$, 均有$\Psi_{l,d}(t)\leq1+C_{l}\Vert t\Vert^{-s(c,\theta)}$成立.
From the above, for any $t$, we have $\Psi_{l,d}(t) \leq 1 + C_{l} \Vert t \Vert^{-s(c,\theta)}$.
\end{proof}
Denote
$$K_{l,q}:=\prod_{j=l}^{q}\big(1-\big(1-(2\ell_{n_{j+1}})^d\big)^{n_{j}-n_{j-1}}\big).$$
\begin{lemma}\label{lem45}
For any $l<q,$ we have 
$$\mathbb{E}(\mu_{l,q}(\mathbb{T}^d))=K_{l,q}\quad\text{and}\quad\mathbb{E}(\mu_{l,q}(\mathbb{T}^d)^{2})\leq K_{l,q}^{2}\int_{\mathbb{T}^d}\int_{\mathbb{T}^d}\Psi_{l,q}(x-y)\mathrm{d}x\mathrm{d}y.$$
\end{lemma}
\begin{proof}
%根据集合 $F_j$的定义可知 $\{F_j\}_{j\geq1}$ 是相互独立的. 因此
Note that $\{F_j\}_{j\geq1}$ are pairwise independent. Therefore,
$$\mathbb{E}(\mu_{l,q}(\mathbb{T}^d))=\prod_{j=l}^{q}\int\int_{\mathbb{T}^d}\mathbbm{1}_{F_{j}}(x)\mathrm{d}x\mathrm{d}\mathbb{P}=\prod_{j=l}^{q}\int_{\mathbb{T}^d}\int\mathbbm{1}_{F_{j}}(x)\mathrm{d}\mathbb{P}\mathrm{d}x.$$
%记$\complement_{F}$ 为$F$的补集. 根据$\{F_j\}_{j\geq1}$的独立性有
Let $\complement_{F}$ be the complement of $F$. It follows from the independence of $\{F_j\}_{j\geq1}$ that for $x\in\mathbb{T}^d$,
\begin{align*}
\int\mathbbm{1}_{\complement_{F_{j}}}(x)\mathrm{d}\mathbb{P}
& =\int\prod_{k=n_{j-1}+1}^{n_{j}}(1-\mathbbm{1}_{B(\omega_k, \ell_{j+1})}(x))\mathrm{d}\mathbb{P}\\
&=\prod_{k=n_{j-1}+1}^{n_{j}}\int(1-\mathbbm{1}_{B(\omega_k, \ell_{j+1})}(x))\mathrm{d}\mathbb{P}\\
&=\prod_{k=n_{j-1}+1}^{n_{j}}(1-(2\ell_{n_{j+1}})^d)=\big(1-(2\ell_{n_{j+1}})^d\big)^{n_{j}-n_{j-1}}.
\end{align*}
Then
$$\int\mathbbm{1}_{F_{j}}(x)\mathrm{d}\mathbb{P}=1-\big(1-(2\ell_{n_{j+1}})^d\big)^{n_{j}-n_{j-1}},$$
which gives that
$$\mathbb{E}\big(\mu_{l,q}(\mathbb{T}^d)\big)=\prod_{j=l}^{q}\big(1-(1-(2\ell_{n_{j+1}})^d)^{n_{j}-n_{j-1}}\big)=K_{l,q}.$$

It follows from Lemma~\ref{lem41}~that
\begin{align*}
&\int\mathbbm{1}_{\complement_{F_{j}}}(x)\mathbbm{1}_{\complement_{F_{j}}}(y)\mathrm{d}\mathbb{P}\\
& =\prod_{k=n_{j-1}+1}^{n_{j}}\int\big(1-\mathbbm{1}_{B(\omega_k, \ell_{j+1})}(x)\big)\big(1-\mathbbm{1}_{B(\omega_k, \ell_{j+1})}(y)\big)\mathrm{d}\mathbb P \\
&\leq\left(1-2(2\ell_{n_{j+1}})^d+(2\ell_{n_{j+1}})^d\mathbbm{1}_{{B(0,2\ell_{n_{j+1}})}}(x-y)\right)^{n_{j}-n_{j-1}} \\
&=:\Phi_{j}(x-y).
\end{align*}

Therefore
\begin{align*}
\mathbb{E}\big(\mu_{l,q} (\mathbb{T}^d)^{2}\big)
&=\int\left(\int_{\mathbb{T}^d}\int_{\mathbb{T}^d}\prod_{j=l}^{q}\mathbbm{1}_{F_{j}}(x)\mathbbm{1}_{F_{j}}(y)\mathrm{d}x\mathrm{d}y\right)\mathrm{d}\mathbb{P} \\
&=\int_{\mathbb{T}^d}\int_{\mathbb{T}^d}\prod_{j=l}^{q}\int\big(1-\mathbbm{1}_{\complement_{F_{j}}}(x)\big)\big(1-\mathbbm{1}_{\complement_{F_{j}}}(y)\big)\mathrm{d}\mathbb{P}\mathrm{d}x\mathrm{d}y \\
&=\int_{\mathbb{T}^d}\int_{\mathbb{T}^d}\prod_{j=l}^{q}\int\big(1-\mathbbm{1}_{\complement_{F_{j}}}(x)-\mathbbm{1}_{\complement_{F_{j}}}(y)+\mathbbm{1}_{\complement_{F_{j}}}(x)\mathbbm{1}_{\complement_{F_{j}}}(y)\big)\mathrm{d}\mathbb{P}\mathrm{d}x\mathrm{d}y \\
&\leq\int_{\mathbb{T}^d}\int_{\mathbb{T}^d}\prod_{j=l}^{q}\big(1-2\big(1-(2\ell_{n_{j+1}})^d\big)^{n_{j}-n_{j-1}}+\Phi_{j}(x-y)\big)\mathrm{d}x\mathrm{d}y.
\end{align*}
%对于 $1-2\big(1-(2r_{n_{j+1}})^d\big)^{n_{j}-n_{j-1}}+\Phi_{j}(x-y)$, 如果$\mathbbm{1}_{B_{\mathbb{T}^d}(0,2r_{n_{j+1}})}(x-y)=0\,,$ 那么
If $\mathbbm{1}_{B(0, 2\ell_{n_{j+1}})}(x-y) = 0\,,$ then
\begin{multline*}
\text{1} -2\big(1-(2\ell_{n_{j+1}})^d\big)^{n_{j}-n_{j-1}}+\Phi_{j}(x-y) \\
=1-2\big(1-(2\ell_{n_{j+1}})^d\big)^{n_{j}-n_{j-1}}+\big(1-2(2\ell_{n_{j+1}})^d\big)^{n_{j}-n_{j-1}} \\
\leq1-2\big(1-(2\ell_{n_{j+1}})^d\big)^{n_{j}-n_{j-1}}+\big(1-(2\ell_{n_{j+1}})^d\big)^{2({n_{j}-n_{j-1}})} \\
=\left(1-\big(1-(2\ell_{n_{j+1}})^d\big)^{n_{j}-n_{j-1}}\right)^{2}.
\end{multline*}
If $\mathbbm{1}_{B(0,2\ell_{n_{j+1}})}(x-y)=1\,,$ then
\begin{multline*}
\text{1} -2\big(1-(2\ell_{n_{j+1}})^d\big)^{n_{j}-n_{j-1}}+\Phi_{j}(x-y) \\
=1-2\big(1-(2\ell_{n_{j+1}})^d\big)^{n_{j}-n_{j-1}}+\big(1-(2\ell_{n_{j+1}})^d\big)^{n_{j}-n_{j-1}} \\
=1-\big(1-(2\ell_{n_{j+1}})^d\big)^{n_{j}-n_{j-1}}.
\end{multline*}
Therefore
\begin{align*}
&\frac{1-2\big(1-(2\ell_{n_{j+1}})^d\big)^{n_{j}-n_{j-1}}+\Phi_{j}(x-y)}{\left(1-(1-(2\ell_{n_{j+1}})^d)^{n_{j}-n_{j-1}}\right)^{2}} \\
&\hspace{5ex}~\hspace{5ex}=\biggl(1+\frac{\big(1-(2\ell_{n_{j+1}})^d\big)^{n_{j}-n_{j-1}}}{1-\big(1-(2\ell_{n_{j+1}})^d\big)^{n_{j}-n_{j-1}}}\mathbbm{1}_{B(0,2\ell_{n_{j+1}})}(x-y)\biggr).
\end{align*}

Since 
\begin{align*}
\Psi_{l,q}(x-y)&=\prod_{j=l}^{q}\biggl(1+\frac{\big(1-(2\ell_{n_{j+1}})^d\big)^{n_{j}-n_{j-1}}}{1-\big(1-(2\ell_{n_{j+1}})^d\big)^{n_{j}-n_{j-1}}}\mathbbm{1}_{B(0,2\ell_{n_{j+1}})}(x-y)\biggr)\\
&=\frac{1}{K_{l,q}^2}\prod_{j=l}^{q}\biggl(1-2\big(1-(2\ell_{n_{j+1}})^d\big)^{n_{j}-n_{j-1}}+\Phi_{j}(x-y)\biggr),
\end{align*}
then
$$\mathbb{E}(\mu_{l,q}(\mathbb{T}^d)^{2})\leq K_{l,q}^{2}\int_{\mathbb{T}^d}\int_{\mathbb{T}^d}\Psi_{l,q}(x-y)\mathrm{d}x\mathrm{d}y.$$
\end{proof}

%对于$0 < s < d$, 定义$\mu$的$s$能量积分为
For $0 < s < d$,  recall that the $s$-potential at a point $x$ due to the measure  $\mu$ is defined by
$$\phi_{s,\mu}(x):=\int\Vert x-y\Vert^{-s}\mathrm{d}\mu(y),$$
and the $s$-energy of $\mu$ is given by
$$I_{s}(\mu):=\int\int\Vert x-y\Vert^{-s}\mathrm{d}\mu(x)\mathrm{d}\mu(y).$$
If $\mu$ is the Lebesgue measure $\lambda$, then 
\begin{equation}\label{estimateofenergy}
I_{s}:=I_{s}(\lambda)=\int_{\mathbb{T}^d}\int_{\mathbb{T}^d}\Vert x-y\Vert^{-s}\mathrm{d}x\mathrm{d}y=\frac{d\cdot 2^s}{d-s}<+\infty.
\end{equation}

By adapting the proof of Lemma 1.1 in Persson \cite{Persson19}, we obtain the following lemma.
\begin{lemma}\label{prop41}
Let $\mu$ be a finite Borel measure with $\emptyset\ne {\rm supp}\mu\subset \T^d$. Suppose that $0<s<d$, and $I_s(\mu)$  is finite. Define a Borel measure by
\[\eta(A)=\int_A\left(\phi_{s,\mu}(x)\right)^{-1}\mathrm{d}\mu(x)\]
for any Borel set $A$. Then the measure $\eta$ satisfies that $\eta(U)\le \Vert U\Vert^s$ for any Borel set $U$.
\end{lemma}

\begin{lemma}\label{prop42}
Let $\epsilon> 0,$ $0 <\delta< 1,$ $\theta > 1$ and $n_j=\theta^j$. If $\ell_n=\frac{c}{n^{\frac{1}{d}}},$ and

$$c>\frac{1}{2}\big(-\frac{\theta^2}{\theta-1}\log(1-\theta^{-1})\big)^\frac{1}{d},$$
then for $l$ large enough, %不等式
\begin{align}\label{with1}
\delta\mathbb{E}\mu_{l,q}(\mathbb{T}^d)<\mu_{l,q}(\mathbb{T}^d)<(2-\delta)\mathbb{E}\mu_{l,q}(\mathbb{T}^d)
\end{align}
holds with probability  at least $1-\epsilon$.
%成立的概率至少为$1-\epsilon$.
\end{lemma}
\begin{proof}
%对$c$的假设意味着 $s(c, \theta) < d\,$.根据引理~\ref{lem43}~和命题~\ref{prop41}~,可得

By  Lemma~\ref{lem43}~ and Lemma~\ref{prop41}~, we obtain that 
\begin{align*}
&\quad\,\, \mathbb{P}\big(|\mu_{l,q}(\mathbb{T}^d)-\mathbb{E}\mu_{l,q}(\mathbb{T}^d)|\geq a\big) \leq\frac{\mathbb{E}\big(\mu_{l,q}(\mathbb{T}^d)-\mathbb{E}\mu_{l,q}(\mathbb{T}^d)\big)^{2}}{a^{2}} \\
&=\frac{\mathbb{E}\big(\mu_{l,q}(\mathbb{T}^d)^{2}\big)-(\mathbb{E}\mu_{l,q}\big(\mathbb{T}^d)\big)^{2}}{a^{2}}\le \frac{K_{l,q}^2}{a^{2}}\left(\int\int\left(\Psi_{l,q}(x-y)-1\right)\mathrm{d}x\mathrm{d}y\right) \\
&\leq \frac{C_lK_{l,q}^2}{a^{2}}\left(\int\int \Vert x-y\Vert^{-s(c,\theta)}\mathrm{d}x\mathrm{d}y\right)=C_lI_{s(c,\theta)}\frac{(\mathbb{E}\mu_{l,q}\big(\mathbb{T}^d)\big)^{2}}{a^{2}}.
\end{align*}
%where
%$$D_{l}:=C_{l}\int_{\mathbb{T}^d}\int_{\mathbb{T}^d}\Vert x-y\Vert^{-s(c,\theta)}\mathrm{d}x\mathrm{d}y= C_{l} J_{s(c,\theta)}.$$
The assumption on $c$ implies that $s(c, \theta) < d\,$, then by inequality~\eqref{estimateofenergy}, there exists $M>0$ such that 
$$C_lI_{s(c,\theta)}<MC_l .$$

Take $a=(1-\delta){\mathbb{E}\mu_{l,q}(\mathbb{T}^d)}$, then
$$\mathbb{P}\Big(\delta<\frac{\mu_{l,q}(\mathbb{T}^d)}{\mathbb{E}\mu_{l,q}(\mathbb{T}^d)}<(2-\delta)\Big)\geq1-\frac{MC_l}{(1-\delta)^{2}}.$$
It follows from Lemma~\ref{lem43}~that $\lim\limits_{l \rightarrow \infty}C_l =0$, which gives this lemma. 
\end{proof}

\begin{lemma}\label{prop43}
Let $\theta > 1$ and $n_j=\theta^j$. If $\ell_n=\frac{c}{n^{\frac{1}{d}}},$  and
$$c>\frac{1}{2}\big(-\frac{\theta^2}{\theta-1}\log(1-\theta^{-1})\big)^\frac{1}{d},$$
Then for $0<s<d-s(c,\theta),$ 
$$\mathbb{E}\big(I_{s}(\mu_{l,q})\big)\leq K_{l,q}^{2}\left(I_s+C_{l}I_{s+s(c,\theta)}\right)<\infty.$$
\end{lemma}
\begin{proof}
Since $c>\frac{1}{2}\big(-\frac{\theta^2}{\theta-1}\log(1-\theta^{-1})\big)^\frac{1}{d},$ we have  $s+s(c,\theta)<d$. It follows from Lemma \ref{lem43} and the proof of Lemma \ref{lem45} that 
\begin{align*}
\mathbb{E}\big(I_{s}(\mu_{l,q})\big)&=\mathbb{E}\left(\int\int\Vert x-y\Vert^{-s}\mathrm{d}\mu_{l,q}(x)\mathrm{d}\mu_{l,q}(y)\right)\\
&=\mathbb{E}\left(\int\int\Vert x-y\Vert^{-s}\prod_{j=l}^q1_{F_j}(x)1_{F_j}(y)\mathrm{d}x\mathrm{d}y\right)\\
& \leq K_{l,q}^{2}\int_{\mathbb{T}^d}\int_{\mathbb{T}^d}\Vert x-y\Vert^{-s}\Psi_{l,q}(x-y)\mathrm{d}x\mathrm{d}y \\
&\leq K_{l,q}^{2}\int_{\mathbb{T}^d}\int_{\mathbb{T}^d}\Vert x-y\Vert^{-s}\left(1+C_{l}\Vert x-y\Vert^{-s(c,\theta)}\right)\mathrm{d}x\mathrm{d}y\\
&=K_{l,q}^{2}(I_{s}+C_{l}I_{s+s(c,\theta)})< \infty.
%&=K_{l,q}^{2}C_{l}\frac{d\cdot 2^{\big(s+s(c,\theta)\big)}}{d-\big(s+s(c,\theta)\big)}< \infty.
\end{align*}
\end{proof}

There follows the proof of the lower bound on the Hausdorff dimension of $\mathcal{U}(\omega,\ell)$ in Theorem~\ref{thm41}.%下面是关于定理~\ref{thm41}~中一致随机覆盖集 Hausdorff 维数下界的证明.
\begin{proof}[Proof of the lower bound in Theorem~{\normalfont\ref{thm41}}]

Since $$\inf_{\theta>1}\frac{1}{2}{\big(-\frac{\theta^2}{\theta-1}\log(1-\theta^{-1})\big)^\frac{1}{d}}=\frac{1}{2},$$ the condition that $c>\frac{1}{2}$ is equivalent to 
\begin{equation}\label{condition:c}
c>\inf_{\theta>1}\frac{1}{2}{\big(-\frac{\theta^2}{\theta-1}\log(1-\theta^{-1})\big)^\frac{1}{d}},
\end{equation}
which is also  equivalent to
$$\sup_{\theta>1}\left(d-s(c, \theta)\right)=\sup_{\theta>1}\left(d+\frac{d\log\big(1-\exp{(-(2c)^d\frac{\theta-1}{\theta^2})}\big)}{\log\theta}\right)>0\,.$$
%{\color{red}Assume that  $\theta$ satisfies 
%\begin{align}\label{to1}
 % c>\frac{\big(-\frac{\theta^2}{\theta-1}\log(1-\theta^{-1})\big)^\frac{1}{d}}{2},
%\end{align}
% equivalent to
%$$d-\frac{-d\log\big(1-e^{(-(2c)^d\frac{\theta-1}{\theta^2})}\big)}{\log\theta}>0\,.$$}
%不失一般性成立.
%由~\eqref{to1}~式, 可以推出 $s(c, \theta) < d$ ,取$s \in \big(0, d -s(c, \theta)\big)$. 设$0 <\epsilon< \frac{1}{2}$ 且 $\delta > 0$.令 $n_j = \theta^j$. 由于$s+s(c,\theta)<d,$ 这意味着 $J_{s+s(c,\theta)}$ 是有限的. 根据Markov不等式和命题~\ref{prop43}~, 对于任意的$a>0$有
Note that $s(c, \theta)$ decreases as $\theta$ increases, then we  deduce that there exists $\theta_0>1$ such that for all $\theta>\theta_0$,
\[d-s(c, \theta)>0.\]
Take $\theta>\theta_0$. Put $n_j = \theta^j$. Let $s \in \big(0, d -s(c, \theta)\big)$, which  implies that $I_{s}$ and $I_{s + s(c, \theta)}$ are finite.
 By the Markov inequality and Lemma ~\ref{prop43}, for any $a > 0$, we have
$$\mathbbm{P}\big(I_{s}(\mu_{l,q})\geq a(\mathbb{E}\mu_{l,q})^{2}\big)\leq\frac{\mathbb{E}I_{s}(\mu_{l,q})}{a(\mathbb{E}\mu_{l,q})^{2}}\leq\frac{I_s+C_{l}I_{s+s(c,\theta)}}{a}.$$

 Let $0 <\varepsilon< \frac{1}{2}$ and $\delta > 0$. Take $a > 0$ with $\frac{I_s+C_{l}I_{s+s(c,\theta)}}{a}<\varepsilon$. Then
$$\mathbb{P}\big(I_{s}(\mu_{l,q})<a(\mathbb{E}\mu_{l,q})^{2}\big)\geq1-\varepsilon.$$
Since inequality \eqref{condition:c} holds, for $l$ large enough, it follows from Lemma~\ref{prop42}~that
\begin{align}\label{to2}
I_{s}(\mu_{l,q})<a(\mathbb{E}\mu_{l,q})^{2}\quad\text{and}\quad\delta<\frac{\mu_{l,q}(\mathbb{T}^d)}{\mathbb{E}\mu_{l,q}(\mathbb{T}^d)}<2-\delta.
\end{align}
holds with probability at least $1-2\varepsilon$.
Therefore for any $q > l$, 
\begin{align}\label{with2}
  \mathbb{P}\left(I_{s}(\mu_{l,q})<\frac{a}{\delta^{2}}\big(\mu_{l,q}(\mathbb{T}^d)\big)^{2}\right)\geq1 -2\varepsilon.
\end{align}
Take $l$ such that ~\eqref{with2}~holds, then
\begin{align*}
  &\mathbb{P}\left(\bigcap_{N=l+1}^{\infty}\bigcup_{q=N}^{\infty}\left(I_{s}(\mu_{l,q})<\frac{a}{\delta^{2}}\big(\mu_{l,q}(\mathbb{T}^d)\big)^{2}\right)\right)\\
  &=\lim_{N\rightarrow\infty}\mathbb{P}\left(\bigcup_{q=N}^{\infty}\left(I_{s}(\mu_{l,q})<\frac{a}{\delta^{2}}\big(\mu_{l,q}(\mathbb{T}^d)\big)^{2}\right)\right)\\
  &\geq1 -2\varepsilon.
\end{align*}
It means that with probability at least $1 -2\varepsilon$, there is a subsequence $(q_i)_i\subset \mathbb{N}$, such that%there exists $\Omega_1\subset \Omega$ such that $\mathbb{P}(\Omega_1)\geq 1-2\epsilon$, and  for all $(\omega_k)_{k\geq 1}\subset\Omega_1$, there exists a sequence $\{q_j\}_{j\geq1}$ such that 
\begin{align} \label{estimate:ismu}
I_{s}(\mu_{l,q_{i}})<\frac{a}{\delta^{2}}\big(\mu_{l,q_{i}}(\mathbb{T}^d)\big)^{2},\quad\forall\, j\geq 1.
\end{align}
%假设 $\{\omega_k\}_{k\geq 1}$ 使得存在序列 $\{q_i\}_{i\geq 1}$. 标准化 $\mu_{l,q_i}$ 定义
Then we normalize the measure $\mu_{l,q_i}$ as 
$$\nu_{l,q_{i}}=\frac{\mu_{l,q_{i}}}{\mu_{l,q_{i}}(\mathbb{T}^d)},$$
which is a probability measure. And we define a measure $\eta_{l,q_{i}}$  by
$$\eta_{l,q_{i}}(U)=\int_U(\phi_{s,\nu_{l,q_{i}}}(x))^{-1}\mathrm{d}(x)$$
for any Borel set $U$. Applying Lemma~\ref{prop41}, we have%是Riesz位势. 根据定理~\ref{thm02}~有
$$\eta_{l,q_{i}}(U)\leq\Vert U\Vert^{s}.$$

It follows from Jensen inequality and inequality \eqref{estimate:ismu} that 
\begin{equation*}
\begin{split}
\eta_{l,q_{i}}(\mathbb{T}^d)&=\int_{\mathbb{T}^d}(\phi_{s,\nu_{l,q_{i}}}(x))^{-1}\mathrm{d}\nu_{l,q_{i}}(x)\\
&\geq \left(\int_{\mathbb{T}^d}\int_{\mathbb{T}^d}\Vert x-y\Vert^{-s}\mathrm{d}\nu_{l,q_{i}}(y)\mathrm{d}\nu_{l,q_{i}}(x)\right)^{-1}\\
&=(\mu_{l,q_{i}}(\mathbb{T}^d))^2\left(\int_{\mathbb{T}^d}\int_{\mathbb{T}^d}\Vert x-y\Vert^{-s}\mathrm{d}\mu_{l,q_{i}}(y)\mathrm{d}\mu_{l,q_{i}}(x)\right)^{-1}\\
&=\frac{(\mu_{l,q_{i}}(\mathbb{T}^d))^2}{I_{s}(\mu_{l,q_{i}})}\geq\frac{\delta^{2}}{a}.
\end{split}
\end{equation*}

Recall that  $\{U_k\}_{k\geq1}$ is an open covering of $F$ defined in \eqref{define:F}, 
%那么对于任意$l$, 特别是足够大被选择的 $l$使得$\{U_k\}_{k\geq1}$覆盖$$\bigcap_{j=l}^{\infty}F_{j},$$ 因为 $F_j$ 是紧集, 存在 $i$ 使得$$\bigcup_{k}U_{k}\supset\bigcap_{j=l}^{q_{i}}F_{j}.$$
then for any $l$, $\{U_k\}_{k\geq1}$ also covers $\bigcap_{j=l}^{\infty}F_{j}.$
Since $\{F_j\}_j$ are  compact, there exists an $i$ such that
$$\bigcap_{j=l}^{q_{i}}F_{j}\subset \bigcup_{k}U_{k}.$$

Since  $\{U_k\}_{k\geq1}$ covers the support of $\eta_{l,q_{i}}$, we deduce that
$$\sum_{k}\Vert U_{k}\Vert^{s}\geq\sum_{k}\eta_{l,q_{i}}(U_{k})=\eta_{l,q_{i}}(\mathbb{T}^d)\geq\frac{\delta^{2}}{a},$$
which implies that $\mathcal{H}^s\left(\bigcap_{j=l}^{q_{i}}F_{j}\right)\geq \frac{\delta^{2}}{a}$. Thus with  probability at least $1-2\varepsilon$, $\mathrm{dim}_{\mathrm{H}}F \geq s$ holds. 
Let $s$ tend to $d-s(c, \theta)$ , then with  probability at least $1-2\varepsilon$, 
$$\mathrm{dim}_{\mathrm{H}}F \geq d-s(c, \theta).$$

Recall that $F\subset \mathcal{U}\left(\omega,\left(\frac{c}{n^{\frac{1}{d}}}\right)\right)$. By the arbitrariness of $\varepsilon$, it follows that almost surely 
%$\epsilon $ 的任意性以及$\mathcal{U}\left(\omega,\left\{\frac{c}{n^{\frac{1}{d}}}\right\}\right)\supset F$, 可以推出几乎必然有
$$\dim_{\mathrm{H}}\mathcal{U}\left(\omega,\left(\frac{c}{n^{\frac{1}{d}}}\right)\right)\geq\sup\limits_{\theta>1}(d-s(c,\theta))=\sup\limits_{\theta>1}\left(d-\frac{-d\log\big(1-e^{(-(2c)^d\frac{\theta-1}{\theta^2})}\big)}{\log\theta}\right).$$

\end{proof}

\subsection*{Acknowledgements}We thank JunJie Huang for useful discussions.

\bibliographystyle{amsplain}

\end{document}